\documentclass[a4paper]{article}

\usepackage[top=1in, bottom=1in, left=1.1in, right=1.8in]{geometry}
\usepackage{amsfonts}
\usepackage{amsmath}
\usepackage{amsthm,amssymb}
\usepackage{CJK,graphicx}
\usepackage{amscd}
\usepackage{amssymb}
\usepackage{mathrsfs}
\usepackage[all,cmtip]{xy}
\usepackage{lmodern}
\usepackage[Symbol]{upgreek}
\usepackage{bm}
\usepackage{eucal}
\usepackage[nopar]{lipsum}
\usepackage{rotating}
\usepackage{tikz}
\usepackage{calc}
\usepackage{tikz-cd}
\usepackage{mathabx}
\usepackage{tikz-3dplot}
\usetikzlibrary{calc}
\usetikzlibrary{decorations.markings,arrows}
\usetikzlibrary{shapes.geometric}
\newsavebox\CBox
\newcommand\hcancel[2][0.5pt]{%
	\ifmmode\sbox\CBox{$#2$}\else\sbox\CBox{#2}\fi%
	\makebox[0pt][l]{\usebox\CBox}%
	\rule[0.5\ht\CBox-#1/2]{\wd\CBox}{#1}}

\newtheorem{theorem}{Theorem}[section]
\newtheorem{corollary}{Corollary}[section]
\newtheorem{definition}{Definition}[section]
\newtheorem{lemma}{Lemma}[section]
\newtheorem{proposition}{Proposition}[section]

\newtheorem{conjecture}{Conjecture}[section]

\newtheorem{remark}{Remark}[section]
\newtheorem{example}{Example}[section]

\AtEndDocument{\bigskip{\footnotesize
	    \textsc{Columbia University}\\
		\textit{E-mail address}: \texttt{yinlee@math.columbia.edu} \par
}}

\begin{document}
	
\title{\textbf{Nonexistence of exact Lagrangian tori in affine conic bundles over $\mathbb{C}^n$}}\author{Yin Li}\date{}\maketitle

\begin{abstract}
Let $M\subset\mathbb{C}^{n+1}$ be a smooth affine hypersurface defined by the equation $xy+p(z_1,\cdots,z_{n-1})=1$, where $p$ is a Brieskorn-Pham polynomial and $n\geq2$. We prove that if $L\subset M$ is a closed, orientable, exact Lagrangian submanifold, then $L$ cannot be a $K(\pi,1)$ space. The key point of the proof is to establish a version of homological mirror symmetry for the wrapped Fukaya category of $M$, from which the finite-dimensionality of the symplectic cohomology group $\mathit{SH}^0(M)$ follows by a Hochschild cohomology computation.
\end{abstract}

\section{Introduction}

\subsection{Overview}

The study of Lagrangian embedding in symplectic manifolds using holomorphic curves dates back to the ground-breaking work of Gromov $\cite{mg}$, where it is proved that for any closed Lagrangian submanifold $L\subset\mathbb{C}^n$, it can never be exact. Using modern languages, Gromov's result can be interpreted as the vanishing of the symplectic cohomology group $\mathit{SH}^\ast(\mathbb{C}^n)$. See $\cite{ps2}$ for a good survey on related topics. 

In general, let $M$ be a Liouville manifold, one can get information about Lagrangian embedding in $M$ by studying the symplectic cohomology $\mathit{SH}^\ast(M)$ or its twisted version. By showing the vanishing of the symplectic cohomology twisted by a generic non-exact 2-form, Ritter proved the following:

\begin{theorem}[$\cite{ar}$]
Let $M$ be a 4-dimensional Milnor fiber of a simple singularity, then $M$ does not contain an exact Lagrangian torus.
\end{theorem}

This forces the closed, orientable exact Lagrangian submanifolds in $M$ to be spheres, and these Lagrangian spheres are later classified when the singularity is type $A_m$ $\cite{ww}$.

One can consider more delicate structures on $\mathit{SH}^\ast(M)$ obtained by counting a 1-parameter family of holomorphic curves, which leads to the definition of the BV (Batalin-Vilkovisky) operator
\begin{equation}
\Delta:\mathit{SH}^\ast(M)\rightarrow\mathit{SH}^{\ast-1}(M).
\end{equation}
A cohomology class $b\in\mathit{SH}^1(M)$ is called a \textit{dilation} by Seidel-Solomon $\cite{ss}$ if its image under the BV operator gives the identity $1\in\mathit{SH}^0(M)$. The compatibility between Viterbo functoriality and the BV operator implies the following:

\begin{theorem}[$\cite{ss}$, Corollary 6.3]\label{theorem:dilation}
Let $M$ be a Liouville manifold admitting a dilation. If $L\subset M$ is a closed exact Lagrangian submanifold, then $L$ cannot be a $K(\pi,1)$ space.
\end{theorem}

\begin{example}\label{example:triple}
It is known, by iterating an argument using Lefschetz fibrations ($\cite{ss}$, Section 7), that any Milnor fiber of an isolated hypersurface singularity obtained by triple stabilizations, i.e. one of the form
\begin{equation}\label{eq:triple}
z_1^2+z_2^2+z_3^2+p(z_4,\cdots,z_{n+1})=0,
\end{equation}
where $p$ is a polynomial with an isolated critical point in the origin of $\mathbb{C}^{n-2}$, admits a dilation. See $\cite{ps4}$, Example 2.13.
\end{example}

More generally, one can replace the identity in the definition of a dilation with an arbitrary class $h\in\mathit{SH}^0(M)$ which is invertible with respect to the pair-of-pants product, and define the notion of a \textit{quasi-dilation} $(h,b)\in\mathit{SH}^0(M)\times\mathit{SH}^1(M)$ by asking $\Delta(hb)=h$. See $\cite{ps5}$, Lecture 19. The following generalization of Theorem \ref{theorem:dilation} is essentially due to Davison:

\begin{theorem}[$\cite{bd}$]\label{theorem:quasi-dilation}
Let $M$ be a Liouville manifold admitting a quasi-dilation. If $L\subset M$ is a closed, orientable, exact Lagrangian submanifold, then $L$ cannot be a hyperbolic space.
\end{theorem}

\begin{example}\label{example:quasi-dilation}
Let $f:\mathbb{C}^{n-1}\rightarrow\mathbb{C}$ be a regular function whose zero set $f^{-1}(0)$ is smooth. Define $M$ to be the affine conic bundle over $(\mathbb{C}^\ast)^{n-1}$:
\begin{equation}\label{eq:GP}
M:=\left\{(x,y,z_1,\cdots,z_{n-1})\in\mathbb{C}^2\times(\mathbb{C}^\ast)^{n-1}|xy=f(z_1,\cdots,z_{n-1})\right\}.
\end{equation}
It is proved in $\cite{gp1}$ that $M$ admits a quasi-dilation. When $\dim_\mathbb{C}(M)=3$, stronger topological restrictions on closed exact Lagrangian submanifolds in $M$ can be established, see $\cite{gp1}$, Theorem 1.2.
\end{example}

Inspired by the work of Davison $\cite{bd}$ on the exactness of Calabi-Yau structures on fundamental group algebras, the author introduced in $\cite{yl2}$ the notion of a \textit{cyclic dilation}. This is a class $\tilde{b}\in\mathit{SH}_{S^1}^1(M)$ in the first degree $S^1$-equivariant symplectic cohomology, whose image under the marking map
\begin{equation}\label{eq:marking}
B:\mathit{SH}^\ast_{S^1}(M)\rightarrow\mathit{SH}^{\ast-1}(M)
\end{equation}
is an invertible element $h\in\mathit{SH}^0(M)^\times$. Similar notions are studied independently by Zhou $\cite{zz2}$. As a more delicate piece of structure than the BV operator, the map $B$ is defined by counting higher-dimensional families of holomorphic cylinders $\cite{bo2}$. It is proved in $\cite{yl2}$ that this condition is strictly weaker than what is required by a quasi-dilation. Thanks to an $S^1$-equivariant version of Viterbo functoriality due to Cohen-Ganatra $\cite{cg}$, we have the following improvement of Theorem \ref{theorem:dilation}:

\begin{theorem}[$\cite{yl2}$, Corollary 5.1]\label{theorem:cyclic}
Let $M$ be a Liouville manifold admitting a cyclic dilation with $h=1$. If $L\subset M$ is a closed exact Lagrangian submanifold, then $L$ cannot be a $K(\pi,1)$ space.
\end{theorem}

\begin{example}\label{example:cyclic}
Let $M\subset\mathbb{C}^{n+1}$ be the Milnor fiber of the isolated singularity
\begin{equation}
z_1^k+\cdots+z_{n+1}^k=0,
\end{equation}
where $n\geq k\geq2$. Using the machinery developed by Diogo-Lisi $\cite{dl}$, it is proved in $\cite{zz2}$ that $M$ admits a cyclic dilation with $h=1$. For the special case when $n=k=3$, see also $\cite{yl2}$, Section 6.1 for an alternative proof.
\end{example}

On can move on and make use of the $L_\infty$-structure on the symplectic cochain complex $\mathit{SC}^\ast(M)$ to produce more delicate obstructions to exact Lagrangian embeddings. This idea, originally due to Seidel, has been carried out recently by Ganatra-Siegel $\cite{gs}$ in the framework of symplectic field theory $\cite{egh}$.

\subsection{New results}

This paper attempts to narrow the gap between the known nonexistence results for the Milnor fibers considered in Examples \ref{example:triple} and \ref{example:cyclic}, and obtain restrictions of exact Lagrangian embeddings for a class of Liouville manifolds to which, due to technical difficulties, neither Theorem \ref{theorem:dilation} nor Theorem \ref{theorem:cyclic} is applicable. More precisely, we shall consider the situation where $M$ is the Milnor fiber of an isolated singularity that is doubly stabilized, namely one defined by the equation
\begin{equation}\label{eq:double}
z_1^2+z_2^2+p(z_3,\cdots,z_{n+1})=0,
\end{equation}
where $p$ is a polynomial which has an isolated singular point at the origin. Note that when $p$ is weighted homogeneous, $M$ is an affine conic bundle over $\mathbb{C}^{n-1}$, whose discriminant locus is the smooth hypersurface defined by $p(z_3,\cdots,z_{n+1})=c$, for some $c\in\mathbb{C}^\ast$. These manifolds have been studied previously by Abouzaid-Auroux-Katzarkov in the context of SYZ (Strominger-Yau-Zaslow) mirror symmetry $\cite{aak}$.

In view of Theorem \ref{theorem:dilation} and Example \ref{example:triple}, it is natural to ask whether there exists a polynomial $p$ (with an isolated critical point at the origin) so that the corresponding Milnor fiber $M$ contains an exact Lagrangian $K(\pi,1)$. Note that for isolated hypersurface singularities that are once stabilized, the construction of Keating $\cite{ak1}$ shows that an exact Lagrangian torus exists in any 4-dimensional Milnor fiber in $\mathbb{C}^3$ defined by the equation
\begin{equation}
z_1^2+z_2^{k_2}+z_3^{k_3}=1,\textrm{ where }\frac{1}{k_2}+\frac{1}{k_3}\leq\frac{1}{2}.
\end{equation}
In fact, a genus two exact Lagrangian surface is expected to exist in the 4-dimensional Milnor fibers above once the more strict condition $\frac{1}{k_2}+\frac{1}{k_3}<\frac{1}{2}$ is imposed.

On the other hand, it follows from $\cite{ps5}$, Corollary 19.7 that any Milnor fiber $M$ of (\ref{eq:double}) admits a quasi-dilation. In particular, it follows from Theorem \ref{theorem:quasi-dilation} that $M$ does not contain any hyperbolic closed, orientable, exact Lagrangian submanifold. In the case when $p$ is a \textit{Brieskorn-Pham polynomial}, namely one of the form
\begin{equation}\label{eq:Brieskorn}
p(z_3,\cdots,z_{n+1})=z_3^{k_3}+\cdots+z_{n+1}^{k_{n+1}},
\end{equation}
where $k_3,\cdots,k_{n+1}$ are positive integers larger than 1, we prove the following strengthened result:

\begin{theorem}\label{theorem:main}
Let $M$ be the Milnor fiber of the singularity (\ref{eq:double}), where $p$ is a Brieskorn-Pham polynomial (\ref{eq:Brieskorn}). If $L\subset M$ is a closed, orientable, exact Lagrangian submanifold, then $L$ cannot be a $K(\pi,1)$ space.
\end{theorem}

\begin{remark}
Note that here (and also in Theorem \ref{theorem:quasi-dilation}) we have restricted ourselves to the situation where $L$ is orientable, while in the statements of Theorems \ref{theorem:dilation} and \ref{theorem:cyclic} we do not. This is because Theorems \ref{theorem:dilation} and \ref{theorem:cyclic} can be proved with the help of Viterbo functoriality $\mathit{SH}^\ast(M)\rightarrow\mathit{SH}^\ast(T^\ast L)$, which holds for non-orientable exact Lagrangian submanifolds $L\subset M$ by considering symplectic cohomologies with local coefficients $\cite{ma2}$. However, for the proofs of Theorems \ref{theorem:quasi-dilation} and \ref{theorem:main}, considerations of open string invariants are required, see $\cite{bd}$, Proposition 5.2.6, and our argument in Lemma \ref{lemma:0}.
\end{remark}

\begin{corollary}\label{corollary:dim3}
Let $M$ be as in Theorem \ref{theorem:main}, and $n=3$. If $L\subset M$ is a closed exact Lagrangian submanifold, then $L$ cannot be diffeomorphic to $S^1\times\Sigma_g$, where $\Sigma_g$ is a closed orientable surface with genus $g\geq1$.
\end{corollary}

This result is optimal in the sense that some of the $3$-dimensional affine conic bundles $M$ contain exact Lagrangian $S^1\times S^2$'s. In particular, $\cite{ak1}$, Proposition 1.2 shows that this is the case as long as $\frac{1}{k_3}+\frac{1}{k_4}\leq\frac{1}{2}$.

\begin{remark}
In fact, we can drop the assumption that $p$ is a Brieskorn-Pham polynomial in dimension 3, thanks to the work of Habermann-Smith $\cite{hs}$. See the discussions in Section \ref{section:cyclic}.
\end{remark}

It would be convenient to give here an outline of the idea behind the proof. Suppose that $L\subset M$ is an exact Lagrangian $K(\pi,1)$ space, then it follows from Theorem \ref{theorem:dilation} that $M$ does not admit a dilation. Since it is known that $M$ admits a quasi-dilation, there must be some invertible element $h\in\mathit{SH}^0(M)^\times$ which differs from the identity. We prove in Lemma \ref{lemma:0} that there are action constraints that $h$ has to satisfy, and as a consequence the existence of such an element would result in the infinite-dimensionality of $\mathit{SH}^0(M)$. So in order to prove Theorem \ref{theorem:main}, it suffices to show that $\mathit{SH}^0(M)$ is finite-dimensional.

For $M$ a Weinstein manifold, Ganatra proves in his thesis $\cite{sg2}$ that the closed-open string map
\begin{equation}\label{eq:CO}
\mathit{CO}:\mathit{SH}^\ast(M)\rightarrow\mathit{HH}^\ast(\mathcal{W}(M))
\end{equation}
is an isomorphism, where the right-hand side is the Hochschild cohomology of the wrapped Fukaya category $\mathcal{W}(M)$, whose construction is due to Abouzaid-Seidel $\cite{as}$. This suggests that symplectic cohomology can be computed algebraically.

For a weighted homogeneous polynomial $w\in\mathbb{K}[z_1,\cdots,z_{n+1}]$ given by
\begin{equation}
w=\sum_{i=1}^{n+1}\prod_{j=1}^{n+1}z_j^{a_{ij}},
\end{equation}
we say that it is \textit{invertible} if the corresponding matrix of powers $A=(a_{ij})$ is invertible. Denote by $\check{w}$ its transpose, which is the weighted homogeneous polynomial specified by the transpose of $A$. Homological mirror symmetry predicts in this case the existence of a triangulated equivalence
\begin{equation}\label{eq:HMS-conj}
D^\mathit{perf}\mathit{MF}(\mathbb{K}^{n+2},\Gamma_w,w+z_0\cdots z_{n+1})\cong D^\mathit{perf}\mathcal{W}(\check{w}^{-1}(1)),
\end{equation}
where $\Gamma_w$ is a finite extension of the multiplicative group $\mathbb{G}_m$ defined by
\begin{equation}\label{eq:group}
\Gamma_w:=\left\{(t_0,\cdots,t_{n+1})\in(\mathbb{G}_m)^{n+2}|t_1^{a_{1,1}}\cdots t_{n+1}^{a_{1,n+1}}=\cdots=t_1^{a_{n+1,1}}\cdots t_{n+1}^{a_{n+1,n+1}}=t_0\cdots t_{n+1}\right\},
\end{equation}
$\mathit{MF}(\mathbb{K}^{n+2},\Gamma_w,w+z_0\cdots z_{n+1})$ is the dg category of $\Gamma_w$-equivariant matrix factorizations of $w+z_0\cdots z_{n+1}$ on $\mathbb{K}^{n+2}$, and $D^\mathit{perf}\mathit{MF}(\mathbb{K}^{n+2},\Gamma_w,w+z_0\cdots z_{n+1})$ is its split-closed derived category. 

This is one of the homological mirror symmetry conjectures concerning the Berglund-H\"{u}bsch transpose of an invertible polynomial, see $\cite{lu1}$, Conjecture 1.4. We prove that (\ref{eq:HMS-conj}) holds when $w$ is of the form (\ref{eq:double}), and $p$ is given by (\ref{eq:Brieskorn}). In particular, $w=\check{w}$ in these situations.

\begin{theorem}\label{theorem:HMS}
For
\begin{equation}\label{eq:conic}
w(z_1.\cdots.z_{n+1})=z_1^2+z_2^2+p(z_3,\cdots,z_{n+1}),
\end{equation}
where $p$ is a Brieskorn-Pham polynomial, we have the following equivalence between $\mathbb{Z}$-graded triangulated categories
\begin{equation}\label{eq:HMS}
D^\mathit{perf}\mathit{MF}(\mathbb{K}^{n+2},\Gamma_w,w+z_0\cdots z_{n+1})\cong D^\mathit{perf}\mathcal{W}(M),
\end{equation}
where $M$ is the Milnor fiber of $w$.
\end{theorem}

\begin{remark}
During the preparation of this paper, the author learned that the $\mathbb{Z}/2$-graded version of $\cite{lu1}$, Conjecture 1.4 has recently been proved by Gammage $\cite{bg}$ via microlocal sheaf calculations. Apart from the $\mathbb{Z}$-grading, which is crucial for our applications, our proof in this special case has the advantage that the wrapped Fukaya category $\mathcal{W}(M)$ can be identified explicitly, and it is shown to be related to the Fukaya category $\mathcal{F}(M)$ of compact Lagrangians via $A_\infty$-Koszul duality. This leads to interesting applications, see Proposition \ref{proposition:PSS} below.
\end{remark}

Combining Theorem \ref{theorem:HMS} and the isomorphism (\ref{eq:CO}), the symplectic cohomology group $\mathit{SH}^\ast(M)$ can be computed as the Hochschild cohomology 
\begin{equation}
\mathit{HH}^\ast\left(\mathit{MF}(\mathbb{K}^{n+2},\Gamma_w,w+z_0\cdots z_{n+1})\right)\cong\mathit{HH}^\ast\left(\mathit{MF}(\mathbb{K}^{n+2},\Gamma_w,w)\right),
\end{equation}
while for the latter we have very nice computational tools recorded in the literature, see Section \ref{section:Hochschild}.

To prove Theorem \ref{theorem:main}, one needs to show that $\mathit{HH}^0\left(\mathit{MF}(\mathbb{K}^{n+2},\Gamma_w,w)\right)$ is finite-dimensional. This is done in Proposition \ref{proposition:HH}. In fact, combining our computation with Koszul duality (Proposition \ref{proposition:Koszul}) also yields the following fact which seems to be of independent interest:

\begin{proposition}\label{proposition:PSS}
Let $M$ be as in Theorem \ref{theorem:HMS}, and assume that the powers $k_3,\cdots,k_{n+1}$ are distinct prime numbers, then the $n$th degree PSS (Piunikhin-Salamon-Schwarz) map $H^n(M;\mathbb{K})\rightarrow\mathit{SH}^n(M)$ is an isomorphism.
\end{proposition}

This paper is organized as follows: in Section \ref{section:HMS}, we compute the wrapped Fukaya category $\mathcal{W}(M)$ using Koszul duality and compare it with the matrix factorization category, which leads to a proof Theorem \ref{theorem:HMS} in Section \ref{section:MF}. Section \ref{section:Hochschild} is devoted to the computation of the Hochschild cohomology group $\mathit{HH}^\ast\left(\mathit{MF}(\mathbb{K}^{n+2},\Gamma_w,w)\right)$. Based on this computation, we give a proof of the main Theorem \ref{theorem:main} in Section \ref{section:argument}. Possible generalizations of Theorem \ref{theorem:main} are discussed in Section \ref{section:cyclic}, the final section of this paper. In Appendix \ref{section:A}, we prove Theorem \ref{theorem:unit}, an algebraic fact which is used in the proof of Theorem \ref{theorem:main}.

\section*{Acknowledgements}
This paper was written during the Covid-19 lockdown in the United Kingdom, when the author was unable to travel to the United States to officially start his postdoc position. The author would like to thank The University of Edinburgh for providing a temporary position and funding support for his research, and various people, especially Mohammed Abouzaid, Sharon Greig, and Nick Sheridan for their efforts to make this possible.

After the paper was written, I realized that the idea that Lemma \ref{lemma:0} should hold in the special case when $L$ is a torus was previously pointed out to me by Daniel Pomerleano during private conversations in 2019. Later I forgot the idea and went on to study other subjects, and then rediscovered it recently when reading $\cite{dt}$. I would like to thank him for the inspiring conversation.

\section{Homological mirror symmetry}\label{section:HMS}

From now on, $\mathbb{K}$ will be a field of characteristic zero. We will be only dealing with Liouville manifolds $M$ with $c_1(M)=0$. Actually, we fix a trivialization of their canonical bundle $K_M$ so that the Fukaya categories are $\mathbb{Z}$-graded $A_\infty$-categories over $\mathbb{K}$. All the Liouville manifolds considered in this paper are finite type, which means they are obtained by attaching a half-infinite cylindrical end $[0,\infty)\times\partial\overline{M}$ to the Liouville domain $\overline{M}$. 

\subsection{Koszul duality}\label{section:Koszul}

Let $M$ be the Milnor fiber of the Brieskorn-Pham singularity
\begin{equation}
z_1^{k_1}+\cdots+z_{n+1}^{k_{n+1}}=0,
\end{equation}
where $k_1,\cdots,k_{n+1}>1$ are positive integers satisfying
\begin{equation}\label{eq:weight}
\frac{1}{k_1}+\cdots\frac{1}{k_{n+1}}\neq1.
\end{equation}
Denote by $\mathcal{F}(M)$ its Fukaya category whose objects consist of closed, exact Lagrangian submanifolds which are \textit{Spin} and graded $\cite{ps9}$, and by $\mathcal{W}(M)$ its wrapped Fukaya category, which also allows certain non-compact Lagrangians as its objects $\cite{as}$.

Let $V_1,\cdots,V_\mu\subset M$ be a basis of vanishing cycles, where
\begin{equation}
\mu:=(k_1-1)(k_2-1)\cdots(k_{n+1}-1)
\end{equation}
is the Milnor number of $M$. By $\cite{ps9}$, Lemmas 4.15 and 4.16, (\ref{eq:weight}) implies that there is a non-zero integer $d$ such that
\begin{equation}\label{eq:Dehn}
(\tau_{V_1}\circ\cdots\circ\tau_{V_\mu})^{k_1\cdots k_{n+1}}=[2d],
\end{equation}
where $\tau_{V_i}$ is the Dehn twist along the Lagrangian sphere $V_i$. Applying Seidel's long exact sequence $\cite{ps10}$, we conclude that $\mathcal{F}(M)$ is split-generated by the objects $V_1,\cdots,V_\mu$. Define
\begin{equation}
\mathcal{F}_M:=\bigoplus_{1\leq i,j\leq\mu}\mathit{CF}^\ast(V_i,V_j)
\end{equation}
to be the Fukaya $A_\infty$-algebra of vanishing cycles. This is a $\mathbb{Z}$-graded proper $A_\infty$-algebra over the semisimple ring $\Bbbk:=\bigoplus_{i=1}^\mu\mathbb{K}e_i$. We have the quasi-equivalence
\begin{equation}\label{eq:F}
\mathcal{F}(M)^\mathit{perf}\cong\mathcal{F}_M^\mathit{perf}
\end{equation}
between triangulated $A_\infty$-categories, where the left-hand side above is the $A_\infty$-category of perfect (right) $A_\infty$-modules over $\mathcal{F}(M)$, while the right-hand side is the $A_\infty$-category of perfect modules over the $A_\infty$-algebra $\mathcal{F}_M$.

On the other hand, as a Weinstein manifold $M$ can be constructed by attaching $n$-handles to the standard $2n$-dimensional closed symplectic disc $D^{2n}$ along a link of $\mu$ Legendrian $(n-1)$-dimensional spheres which are unknots along the boundary $S^{2n-1}$. To see this, one can start with the Lefschetz fibration $\tilde{f}:D^{2n}\rightarrow D^2$ defined by the Morsification of the Brieskorn-Pham polynomial
\begin{equation}
f(z_1,\cdots,z_n):=z_1^{k_1}+\cdots+z_n^{k_n}.
\end{equation}
This Lefschetz presentation of $D^{2n}$ describes it as the result of attaching $\frac{\mu}{k_{n+1}-1}$ Weinstein $n$-handles along the (vertical) boundary of $F\times D^2$, where $F$ is the Milnor fiber of the isolated singularity $\{f=0\}\subset\mathbb{C}^n$. Starting from $\tilde{f}$, one can construct another Lefschetz fibration $\pi:\overline{M}\rightarrow D^2$, whose smooth fiber is symplectomorphic to $F$. See for example $\cite{ak1}$, Section 2.5. The Lefschetz fibration $\pi$ has $\frac{k_{n+1}\cdot\mu}{k_{n+1}-1}$ critical values, which are divided into $k_{n+1}$ groups, with each group containing $\frac{\mu}{k_{n+1}-1}$ critical values. Figure \ref{fig:base} gives an illustration of the base of the Lefschetz fibration $\pi$ in the case when $n=3$, and $k_1=k_2=k_3=k_4=3$. Take any one of the groups of $\frac{\mu}{k_{n+1}-1}$ critical values, and assume that they are contained in a half-disc $D_-\subset D^2$, with the other $k_{n+1}-1$ groups of critical values contained in the complement $D^2\setminus D_-$. The preimage $\pi^{-1}(D_-)\subset\overline{M}$, with corners rounded off, is Weinstein deformation equivalent to $D^{2n}$. Now the presentation given by $\pi$ describes $\overline{M}$ as the result of attaching $\mu$ critical handles to $\pi^{-1}(D_-)$ along some Legendrian link $\Lambda\subset\partial\pi^{-1}(D_-)$. For each of the $\mu$ critical values lying outside of $D_-$, there is a matching sphere $V_\bullet\subset M$ whose matching path connects the critical value to a critical value inside $D_-$. Denote by $\star\in D^2$ the origin of the base, the intersection $V_\bullet\cap\pi^{-1}(\star)$ is an unknotted Legendrian sphere $\Lambda_\bullet\subset\partial\pi^{-1}(D_-)$, which defines a connected component of $\Lambda$. In Figure \ref{fig:base}, such a sphere is depicted, and its intersection $\overline{V}_\bullet$ with $\pi^{-1}(D_-)$ is a Lagrangian disc in $D^{2n}$ which fills the unknot $\Lambda_\bullet$. The matching spheres $\{V_\bullet\}$ form a basis of vanishing cycles in $M$, and we will denote them by $V_1,\cdots,V_\mu$, and the corresponding Legendrian spheres by $\Lambda_1,\cdots,\Lambda_\mu$.

\begin{figure}
	\centering
	\begin{tikzpicture}
		\node at (-1.4,3) {$\times$};
		\node at (-1,3) {$\times$};
		\node at (-0.6,3) {$\times$};
		\node at (-0.2,3) {$\times$};
		\node at (0.2,3) {$\times$};
		\node at (0.6,3) {$\times$};
		\node at (1,3) {$\times$};
		\node at (1.4,3) {$\times$};
		
		\node at (-1.898,-2.711) {$\times$};
		\node at (-2.098,-2.365) {$\times$};
		\node at (-2.298,-2.019) {$\times$};
		\node at (-2.498,-1.673) {$\times$};
		\node at (-2.698,-1.327) {$\times$};
		\node at (-2.898,-0.981) {$\times$};
		\node at (-3.098,-0.635) {$\times$};
		\node at (-3.298,-0.289) {$\times$};
		
		\node at (1.898,-2.711) {$\times$};
		\node at (2.098,-2.365) {$\times$};
		\node at (2.298,-2.019) {$\times$};
		\node at (2.498,-1.673) {$\times$};
		\node at (2.698,-1.327) {$\times$};
		\node at (2.898,-0.981) {$\times$};
		\node at (3.098,-0.635) {$\times$};
		\node at (3.298,-0.289) {$\times$};
		
		\node at (0,0) {$\star$};
		\draw [orange,dashed] (-2.8,2.8) to [in=120,out=-60] (0,0);
		\draw [orange,dashed] (0,0) to [in=120,out=-60] (1.6,-3.6);
		\draw (0,0) circle [radius=4];
		\draw [blue] (1,3) to [in=0,out=-90] (0,0);
		\draw [blue] (0,0) to [in=0,out=180] (-3.098,-0.635);
		\node [blue] at (1.2,1) {$V_\bullet$};
		\node [blue] at (-1.5,-0.6) {$\overline{V}_\bullet$};
		\node at (-3,1) {$D_-$};
	\end{tikzpicture}
	\caption{Base of the Lefschetz fibration $\pi:M\rightarrow\mathbb{C}$ when $n=3$ and $k_1=k_2=k_3=k_4=3$.}
	\label{fig:base}
\end{figure}
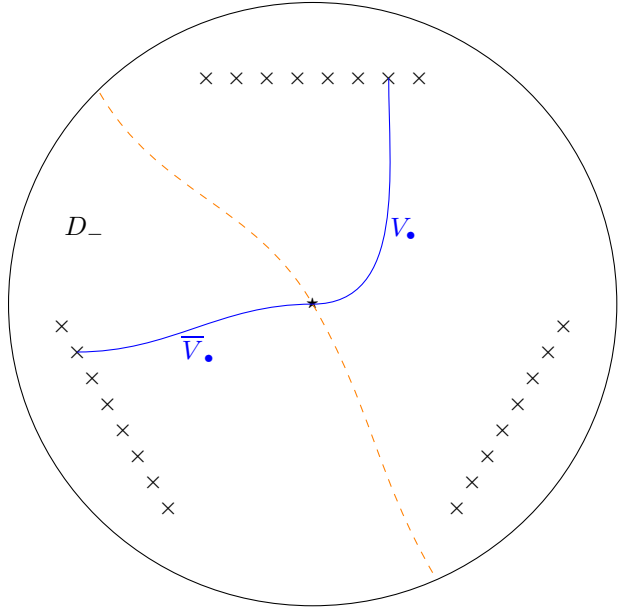

Denote by $\overline{L}_1,\cdots,\overline{L}_\mu\subset\overline{M}$ the Lagrangian cocore discs of the Weinstein $n$-handles whose attaching spheres are $\Lambda_1,\cdots,\Lambda_\mu\subset(S^{2n-1},\xi_\mathit{std})$ respectively, where $\xi_\mathit{std}$ denotes the standard contact structure on the sphere. They can be extended to non-compact Lagrangian submanifolds $L_1,\cdots,L_\mu\subset M$ by adding an infinite cone $[0,\infty)\times\partial\overline{L}_i$, where $1\leq i\leq\mu$, which makes them objects of the wrapped Fukaya category. In particular $V_i\cap L_j\neq\emptyset$ only when $i=j$, in which case $V_i$ intersects $L_i$ transversely at a unique point. Define the wrapped Fukaya $A_\infty$-algebra to be
\begin{equation}
\mathcal{W}_M:=\bigoplus_{1\leq i,j\leq\mu}\mathit{CW}^\ast(L_i,L_j),
\end{equation}
where $\mathit{CW}^\ast$ is the wrapped Floer cochain complex. It is a theorem due to Chantraine-Dimitroglou-Rizell-Ghiggini-Golovko $\cite{cdgg}$ and Ganatra-Pardon-Shende $\cite{gps2}$ that $L_1,\cdots,L_\mu$ generate the wrapped Fukaya category $\mathcal{W}(M)$, from which we get the quasi-equivalence
\begin{equation}\label{eq:W}
\mathcal{W}(M)^\mathit{perf}\cong\mathcal{W}_M^\mathit{perf}.
\end{equation}

Note that both of the $A_\infty$-algebras $\mathcal{F}_M$ and $\mathcal{W}_M$ are augmented. For $\mathcal{F}_M$, the augmentation $\mathcal{F}_M\rightarrow\Bbbk$ is the trivial projection to the idempotents in degree 0. According to Bourgeois-Ekhlom-Eliashberg $\cite{bee}$, Legendrian surgery gives a quasi-isomorphism
\begin{equation}\label{eq:BEE}
\mathcal{W}_M\cong\mathit{CE}^\ast(\Lambda),
\end{equation}
where $\mathit{CE}^\ast(\Lambda)$ is the Chekanov-Eliashberg dg algebra of the Legendrian link $\Lambda\subset(S^{2n-1},\xi_\mathit{std})$. Since $\Lambda$ has $\mu$ connected components, $\mathit{CE}^\ast(\Lambda)$ possesses the structure of a dg algebra over $\Bbbk$. We define the augmentation on $\mathcal{W}_M$ to be the composition
\begin{equation}
\mathcal{W}_M\xrightarrow{\cong}\mathit{CE}^\ast(\Lambda)\rightarrow\Bbbk,
\end{equation}
where the augmentation on $\mathit{CE}^\ast(\Lambda)$ is induced by the Lagrangian fillings $\bigsqcup_{1\leq i\leq\mu}\overline{V}_i$ of $\Lambda$. Our goal in this subsection is to establish the following Koszul duality result:

\begin{proposition}\label{proposition:Koszul}
There are quasi-isomorphisms between augmented $A_\infty$-algebra over $\Bbbk$:
\begin{equation}\label{eq:Koszul}
R\hom_{\mathcal{W}_M}(\Bbbk,\Bbbk)\cong\mathcal{F}_M,R\hom_{\mathcal{F}_M}(\Bbbk,\Bbbk)\cong\mathcal{W}_M.
\end{equation}
\end{proposition}

\begin{proof}
The argument is an adaptation of that appeared in $\cite{yl2}$, Section 6.1 to the more general case, which, for the sake of self-containedness we will reproduce here.

A generalized version of the Eilenberg-Moore equivalence, established in $\cite{ekl}$, Theorem 4, states that we have the following quasi-isomorphism
\begin{equation}
R\hom_{\mathit{CE}^\ast(\Lambda)}(\Bbbk,\Bbbk)\cong\mathcal{F}_M.
\end{equation}
Combining with (\ref{eq:BEE}), we have proved the first quasi-isomorphism in (\ref{eq:Koszul}).

For the second quasi-isomorphism, we use the assumption (\ref{eq:weight}). By $\cite{lu1}$, Proposition 6.5 and Lemma 6.6, (\ref{eq:Dehn}) implies that the diagonal bimodule over $\mathcal{W}(M)$ is a colimit of $\mathcal{F}(M)\otimes\mathcal{F}(M)^\mathit{op}$ in the category of $\mathcal{W}(M)$-bimodules, where $\mathcal{F}(M)^\mathit{op}$ is the opposite category of $\mathcal{F}(M)$. Beilinson's diagonal argument $\cite{abe}$ then implies that there is a fully faithful embedding
\begin{equation}\label{eq:emb1}
\mathcal{W}(M)\hookrightarrow\mathcal{F}(M)^\mathit{mod}
\end{equation}
of the wrapped Fukaya category into the $A_\infty$-category of right $A_\infty$-modules over $\mathcal{F}(M)$. Observe that for any closed Lagrangian submanifold $K\subset M$ which defines an object of $\mathcal{W}(M)$ (e.g. when $K$ is tautologically unobstructed or monotone), the wrapped Floer cohomology $\mathit{HW}^\ast(L,K)$ is finite-dimensional for any object $L$ of $\mathcal{W}(M)$. This shows that the fully faithful embedding (\ref{eq:emb1}) actually gives rise to a cohomologically full and faithful $A_\infty$-functor
\begin{equation}
\mathcal{I}:\mathcal{W}(M)^\mathit{perf}\hookrightarrow\mathcal{F}(M)^\mathit{prop},
\end{equation}
where $\mathcal{F}(M)^\mathit{prop}$ is the $A_\infty$-category of proper $A_\infty$-modules over $\mathcal{F}(M)$. Using the quasi-equivalences (\ref{eq:F}) and (\ref{eq:W}), and passing to the derived functor, we get a fully faithful functor
\begin{equation}\label{eq:emb2}
D\mathcal{I}:D^\mathit{perf}(\mathcal{W}_M)\hookrightarrow D^\mathit{prop}(\mathcal{F}_M).
\end{equation}
By $\cite{sp}$, Proposition 13.34.6, a compact object $X$ is a weak generator of a compactly generated triangulated category $\mathcal{D}$ with arbitrary direct sums if and only if $X$ split-generates the triangulated subcategory of compact objects $\mathcal{D}_c$ of $\mathcal{D}$. In order to apply this fact to the context here, notice that $D^\mathit{perf}(\mathcal{F}_M)$ is a full triangulated subcategory of $D^\mathit{prop}(\mathcal{F}_M)$, and it is actually the subcategory of compact objects of $D^\mathit{prop}(\mathcal{F}_M)$.

Consider the basis of vanishing cycles $V_1,\cdots,V_\mu\subset M$, now regarded as objects of $\mathcal{W}(M)$. It follows from the split-generation of $\mathcal{F}(M)$ by $V_1,\cdots,V_\mu$ that under the functor $D\mathcal{I}$, their Yoneda modules define compact weak generators of the triangulated category $D^\mathit{prop}(\mathcal{F}_M)$. It follows from $\cite{vl}$, Theorem 2.2 that $D^\mathit{prop}(\mathcal{F}_M)$ itself is in fact the smallest triangulated subcategory of $D^\mathit{prop}(\mathcal{F}_M)$ containing the Yoneda modules of $V_1,\cdots,V_\mu$ and closed under direct sums. Since $D^\mathit{perf}(\mathcal{W}_M)$ is obviously such a subcategory of $D^\mathit{prop}(\mathcal{F}_M)$ by (\ref{eq:emb2}), we obtain a fully faithful embedding $D^\mathit{prop}(\mathcal{F}_M)\hookrightarrow D^\mathit{perf}(\mathcal{W}_M)$ in the reverse direction. In this way, we have proved the quasi-equivalence
\begin{equation}
\mathcal{W}_M^\mathit{perf}\cong\mathcal{F}_M^\mathit{prop}.
\end{equation}
Recall from $\cite{ekl}$ that the \textit{completed} Chekanov-Eliashberg dg algebra of the Legendrian link $\Lambda$ is by definition
\begin{equation}\label{eq:w-Koszul}
\widehat{\mathit{CE}}^\ast(\Lambda):=(\mathrm{B}\mathcal{F}_M)^\#,
\end{equation}
where the right-hand side is the linear dual of the bar construction of the augmented $A_\infty$-algebra $\mathcal{F}_M$. Combining (\ref{eq:w-Koszul}) with the Koszul duality functor $R\hom_{\mathcal{F}_M}(\cdot,\Bbbk)$, we get a quasi-equivalence
\begin{equation}
\mathcal{W}_M^\mathit{perf}\cong\widehat{\mathit{CE}}^\ast(\Lambda)^\mathit{perf}.
\end{equation}
In view of the surgery quasi-isomorphism (\ref{eq:BEE}), it follows that
\begin{equation}
\mathit{CE}^\ast(\Lambda)\cong\widehat{\mathit{CE}}^\ast(\Lambda),
\end{equation}
which shows that $\mathit{CE}^\ast(\Lambda)$ is complete as a dg algebra, and therefore (\ref{eq:w-Koszul}) gives the second quasi-equivalence in (\ref{eq:Koszul}).
\end{proof}

\begin{remark}
The Koszul duality between compact and wrapped Fukaya $A_\infty$-algebras of Weinstein manifolds have been studied in many situations since the work of Etg\"{u}-Lekili $\cite{etl1}$. In the case when $M$ is a smooth affine variety which contains a set of compact Lagrangian submanifolds intersecting the cocores in a nice way (e.g. Milnor fibers), it is expected to be true when $M$ is log general type or has log Kodaira dimension $-\infty$. For the log general type case, see $\cite{lp1,lu1}$, for the case when the log Kodaira dimension of $M$ is $-\infty$, see $\cite{ekl,etl1,yl1}$.
\end{remark}

\subsection{Fukaya categories}

From now on we will work in the more restrictive case of affine conic bundles. Let $M$ be the Milnor fiber of the isolated hypersurface singularity defined by
\begin{equation}
w(z_1,\cdots,z_{n+1})=0,
\end{equation}
where $w$ is as in (\ref{eq:conic}). We compute in this subsection the compact and the wrapped Fukaya categories of $M$.

For the compact Fukaya category $\mathcal{F}(M)$, we notice that $M$ is by definition the smooth fiber of a Lefschetz fibration $\tilde{w}:\mathbb{C}^{n+1}\rightarrow\mathbb{C}$, which is the double suspension of the Lefschetz fibration $\tilde{p}:\mathbb{C}^{n-1}\rightarrow\mathbb{C}$ obtained by Morsifying the polynomial $p(z_3,\cdots,z_{n+1})$ as in (\ref{eq:Brieskorn}). Denote by $\mathcal{F}(\tilde{p})$ the Fukaya category associated to the Lefschetz fibration $\tilde{p}$, whose construction is due to Seidel $\cite{ps1}$. Notice that we have the quasi-equivalence $\mathcal{F}(\tilde{p})\cong\mathcal{F}(\tilde{w})$. It follows from $\cite{ps6}$, Corollary 3.1 that there is a quasi-equivalence
\begin{equation}\label{eq:suspension}
\mathcal{V}(M)\cong\mathcal{F}(\tilde{p})\oplus\mathcal{F}(\tilde{p})^\vee[-n],
\end{equation}
where $\mathcal{V}(M)\subset\mathcal{F}(M)$ is the full $A_\infty$-subcategory of the Fukaya category which is formed by a basis of vanishing cycles, and $\mathcal{F}(\tilde{p})^\vee$ is the dual of its diagonal bimodule. Passing to the endomorphism algebras on both sides, we obtain a quasi-isomorphism
\begin{equation}\label{eq:trivial-ext}
\mathcal{F}_M\cong\mathcal{F}_{\tilde{p}}\oplus\mathcal{F}_{\tilde{p}}^\vee[-n]
\end{equation}
between the compact Fukaya $A_\infty$-algebra and the trivial extension (also known as \textit{cyclic completion}, in the terminology of Segal $\cite{es}$) of the endomorphism algebra of Lefschetz thimbles of $\tilde{p}$ by the dual of its diagonal bimodule, where $\mathcal{F}_{\tilde{p}}$ is the $A_\infty$-algebra over $\Bbbk$ defined by
\begin{equation}
\mathcal{F}_{\tilde{p}}:=\bigoplus_{1\leq i,j\leq\mu}\mathit{CF}^\ast(\Delta_i,\Delta_j),
\end{equation}
where $\Delta_1,\cdots,\Delta_\mu\subset\mathbb{C}^{n-1}$ is a basis of Lefschetz thimbles. The labellings are arranged so that when we consider the lifts $\widetilde{\Delta}_1,\cdots,\widetilde{\Delta}_\mu\subset\mathbb{C}^{n+1}$ of these thimbles as the Lefschetz thimbles of $\tilde{w}$, their restrictions to a fiber at infinity are Hamiltonian isotopic to the vanishing cycles $V_1,\cdots,V_\mu$, whose endomorphism algebra defines the Fukaya $A_\infty$-algebra $\mathcal{F}_M$. The Floer complex $\mathit{CF}^\ast(\Delta_i,\Delta_j)$ can be defined using a Hamiltonian perturbation which is small at infinity $\cite{ps11}$, or one can regard it as a partially wrapped Floer cochain complex $\cite{zs}$.

In our case, the basis of Lefschetz thimbles can be chosen so that the directed $A_\infty$-algebra $\mathcal{F}_{\tilde{p}}$ admits a very simple description. Denote by $\mathcal{A}_m$ the path algebra of the $A_m$ quiver, see Figure \ref{fig:Am}, with the arrows located in degree 1, the quasi-isomorphism
\begin{equation}\label{eq:tensor}
\mathcal{F}_{\tilde{p}}\cong\mathcal{A}_{k_3-1}\otimes_\Bbbk\cdots\otimes_\Bbbk\mathcal{A}_{k_{n+1}-1}
\end{equation}
between (formal) $A_\infty$-algebras over $\Bbbk$ is proved in $\cite{fu}$, Theorem 1.1.

\begin{figure}
	\centering
	\begin{tikzpicture}
		\node[circle,draw, fill, minimum size = 2pt,inner sep=1pt] at (0,0) {};
		\node[circle,draw, fill, minimum size = 2pt,inner sep=1pt] at (8,0) {};
		\node[circle,draw, fill, minimum size = 2pt,inner sep=1pt] at (1.5,0) {};
		\node[circle,draw, fill, minimum size = 2pt,inner sep=1pt] at (3,0) {};
		
		\draw[->,shorten >=8pt, shorten <=8pt] (0,0) to (1.5,0);
		\draw[->,shorten >=8pt, shorten <=8pt] (1.5,0) to (3,0);
		
		\path (3,0) to node {\dots} (8,0);
		\node [shape=circle,minimum size=2pt, inner sep=1pt] at (4.5,0) {};
		\draw[->,shorten >=8pt, shorten <=8pt] (3,0) to (4.5,0);
		
		\node [shape=circle,minimum size=2pt, inner sep=1pt] at (6.5,0) {};
		\draw[->,shorten >=8pt, shorten <=8pt] (6.5,0) to (8,0);
		
		\node at (0,0.25) {$1$};
		\node at (1.5,0.25) {$2$};
		\node at (3,0.25) {$3$};
		\node at (8,0.25) {$m$};
	\end{tikzpicture}
	\caption{The $A_m$ quiver}
	\label{fig:Am}
\end{figure}
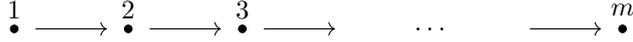
\bigskip

To compute the wrapped Fukaya category $\mathcal{W}(M)$, we use Proposition \ref{proposition:Koszul}, which says that up to quasi-isomorphism, the wrapped Fukaya $A_\infty$-algebra can be computed in terms of the compact Fukaya $A_\infty$-algebra:
\begin{equation}\label{eq:bar}
\mathcal{W}_M\cong(\mathrm{B}\mathcal{F}_M)^\#.
\end{equation}

Let $\mathcal{A}$ be a dg algebra over $\Bbbk$. Recall that the $n$-\textit{Calabi-Yau completion} of $\mathcal{A}$, whose construction is due to Keller $\cite{bk,bke}$, is the tensor dg algebra
\begin{equation}
\Pi_n(\mathcal{A}):=T_\mathcal{A}\left(\mathcal{A}^\dag[n-1]\right),
\end{equation}
where $\mathcal{A}^\dag:=R\hom_{\mathcal{A}^e}(\mathcal{A},\mathcal{A}^e)$ is the derived dual of the semi-free resolution of the diagonal bimodule $\mathcal{A}$ (which we still denote by $\mathcal{A}$ by abuse of notation), and $\mathcal{A}^e:=\mathcal{A}\otimes\mathcal{A}^\mathit{op}$. It is proved in $\cite{bk}$, Theorem 4.8 (see also $\cite{bke}$, Theorem 1.1) that if $\mathcal{A}$ is homologically smooth, then $\Pi_n(\mathcal{A})$ is an \textit{exact} $n$-Calabi-Yau algebra. Here, the exactness of a Calabi-Yau structure means that the class in $\mathit{HH}_{-n}\left(\Pi_n(\mathcal{A})\right)$\footnote{We use the cohomological grading on the Hoshschild chain complex.} which defines the (weak) smooth Calabi-Yau structure admits a lift to the (positive) cyclic homology group $\mathit{HC}_{-n+1}\left(\Pi_n(\mathcal{A})\right)$.

Moreover, we have the following:

\begin{proposition}[$\cite{hlw}$, Theorem 6]\label{proposition:trivial-ext}
Let $\mathcal{A}$ be a proper dg algebra over $\Bbbk$ which is complete, then the Koszul dual of its cyclic completion $\mathcal{A}\oplus\mathcal{A}^\vee[-n]$ is quasi-isomorphic to the $n$-Calabi-Yau completion $\Pi_n(\mathcal{A}^!)$, where $\mathcal{A}^!$ is the Koszul dual of $\mathcal{A}$, which, by our assumptions, must be homologically smooth.
\end{proposition}

We now prove the main result of this subsection:

\begin{proposition}\label{proposition:CY-completion}
We have the following quasi-equivalence between triangulated categories:
\begin{equation}\label{eq:CY-completion}
D^\mathit{perf}\left(\mathcal{W}_M\right)\cong D^\mathit{perf}\left(\Pi_n(\mathcal{F}_{\tilde{p}})\right).
\end{equation}
\end{proposition}

\begin{proof}
By (\ref{eq:trivial-ext}) and (\ref{eq:bar}), we see that $\mathcal{W}_M$ is quasi-isomorphic to the Koszul dual of the cyclic completion of $\mathcal{F}_{\tilde{p}}$. In order to apply Proposition \ref{proposition:trivial-ext}, we only need to verify that $\mathcal{F}_{\tilde{p}}$ is a complete dg algebra, since its properness follows from the definition. To do this, we appeal to (\ref{eq:tensor}). It is clear that the tensor algebra $\mathcal{A}_{k_3-1}\otimes_\Bbbk\cdots\otimes_\Bbbk\mathcal{A}_{k_{n+1}-1}$ with vanishing differential is complete since the path length filtrations on each $\mathcal{A}_{k_i-1}$, where $3\leq i\leq n+1$, induce a filtration on $\mathcal{A}_{k_3-1}\otimes_\Bbbk\cdots\otimes_\Bbbk\mathcal{A}_{k_{n+1}-1}$ which is Hausdorff and complete. Proposition \ref{proposition:trivial-ext} then implies the quasi-isomorphism
\begin{equation}
\mathcal{W}_M\cong\Pi_n(\mathcal{F}_{\tilde{p}}^!).
\end{equation}
The proposition now follows from the fact that the (formal) $A_\infty$-algebra $\mathcal{F}_{\tilde{p}}$ is self-Koszul dual (in the derived sense), which is a consequence of (\ref{eq:tensor}). Alternatively, this can be seen by identifying $\mathcal{F}_{\tilde{p}}^!$ with the endomorphism algebra of a basis of dual Lefschetz thimbles $\Delta_1^!,\cdots,\Delta_\mu^!$, where $\Delta_i\cap\Delta_j^!=\emptyset$ unless $i=j$, in which case $\Delta_i$ intersects $\Delta_i^!$ transversely at a unique point. See $\cite{fss}$, Section 6e for details.
\end{proof}

\begin{remark}
Proposition \ref{proposition:CY-completion} partially proves a conjecture of Lekili-Ueda. In $\cite{lu2}$, Remark 3.1, it is conjectured that (\ref{eq:CY-completion}) should hold in the more general case when $p$ is an arbitrary invertible weighted homogeneous polynomial. They stated the conjecture after verifying that (\ref{eq:CY-completion}) holds when $\{p=0\}\subset\mathbb{C}^{n-1}$ is a simple singularity. Similar results were also obtained previously for Milnor fibers $M$ which are not necessarily associated to doubly stabilized or weighted homogeneous singularities. See, for example $\cite{yl1}$, Theorem 1.2.
\end{remark}

\subsection{Matrix factorizations}\label{section:MF}

We start with a recap of the materials in $\cite{lu2}$, because of this the exposition here will be quite sketchy.

Let $\Gamma$ be a finite extension of $\mathbb{G}_m$, so its character group $\widecheck{\Gamma}:=\hom(\Gamma,\mathbb{G}_m)$  is a finite extension of $\mathbb{Z}$. Assume that $\Gamma$ embeds in $(\mathbb{G}_m)^{n+1}$ as a subgroup, so that it acts diagonally on the affine space $\mathbb{K}^{n+1}$. The action of $\Gamma$ induces a $\widecheck{\Gamma}$-grading on the coordinate ring $\mathbb{K}[z_1,\cdots,z_{n+1}]$. Define $\chi_i=\deg(z_i)$ for $i=1,\cdots,n+1$. Let $w\in\mathbb{K}[z_1,\cdots,z_{n+1}]_\chi$ be a polynomial which is homogeneous of degree $\chi\in\widecheck{\Gamma}$. Given this data, one can define a dg category $\mathit{MF}(\mathbb{K}^{n+1},\Gamma,w)$, the category of $\Gamma$-equivariant matrix factorizations of $w$. Denote by $(\cdot)$ downward shift of the weight grading. The category $\mathit{MF}(\mathbb{K}^{n+1},\Gamma,w)$ is split-generated by the object $E=\bigoplus_{\rho\in\Xi}\mathcal{O}_0(\rho)$, where $\mathcal{O}_0$ is the structure sheaf of the origin, and $\Xi\subset\widecheck{\Gamma}$ is a set of representatives of the quotient group $\widecheck{\Gamma}/(\chi)$. We assume that $\Xi$ is finite. Denote by $\mathcal{A}$ the endomorphism algebra of $E$ in $\mathit{MF}(\mathbb{K}^{n+1},\Gamma,w)$.

Introducing an additional variable $z_0$ with degree $\chi-\sum_{i=1}^{n+1}\chi_i$ allows us to define another dg category of $\Gamma$-equivariant matrix factorizations, $\mathit{MF}(\mathbb{K}^{n+2},\Gamma,w)$. This dg category is split-generated by the object $\widetilde{E}:=E\otimes\mathbb{K}[z_0]$. It is proved in $\cite{lu2}$, Section 4 that the endomorphism algebra of $\widetilde{E}$ is quasi-isomorphic to the $n$-Calabi-Yau completion $\Pi_n(\mathcal{A})$.

Now we take the group $\Gamma$ to be $\Gamma_w$, which is defined in (\ref{eq:group}), and assume that $w$ is as in (\ref{eq:conic}). Since $w$ is itself a Brieskorn-Pham polynomial, it follows from $\cite{fu}$, Theorem 1.3 that there is an equivalence between triangulated categories
\begin{equation}\label{eq:HMS-directed}
D^\mathit{perf}\mathcal{F}(\tilde{p})\cong D^\mathit{perf}\mathit{MF}(\mathbb{K}^{n+1},\Gamma_w,w),
\end{equation}
which establishes a version of homological mirror symmetry for these singularities. In fact, the original statement of $\cite{fu}$, Theorem 1.3 uses the graded version of the triangulated category of singularities instead of the triangulated category of equivariant matrix factorizations. However, according to $\cite{pv}$, Theorem 3.14, these two categories are equivalent as triangulated categories.

By Proposition \ref{proposition:CY-completion}, the equivalence (\ref{eq:HMS-directed}) implies the derived equivalence
\begin{equation}
D^\mathit{perf}\mathcal{W}(M)\cong D^\mathit{perf}\mathit{MF}(\mathbb{K}^{n+1},\Gamma_w,w).
\end{equation}
In order to reach the desired equivalence (\ref{eq:HMS}), we use $\cite{agv}$, Theorem 12.6. In our case, the appearance of the two quadratic terms in the definition of $w$ ensures that the monomial $z_0\cdots z_{n+1}$ lies above the diagonal of $w$, so $w+z_0\cdots z_{n+1}$ is right equivalent to $w$ by a formal coordinate change. Since this coordinate change can be chosen to be $\Gamma_w$-equivariant, we have the quasi-equivalence
\begin{equation}
\mathit{MF}(\mathbb{K}^{n+1},\Gamma_w,w)\cong\mathit{MF}(\mathbb{K}^{n+1},\Gamma_w,w+z_0\cdots z_{n+1}),
\end{equation}
which completes the proof of Theorem \ref{theorem:HMS}.

\begin{remark}\label{remark:LG}
As an affine conic bundle over the affine toric variety $\mathbb{C}^{n-1}$, $M$ admits another (equidimensional) Landau-Ginzburg model $(Y_0,w_0)$ as its mirror, where $Y_0$ is the complement of a smooth hypersurface in a toric Calabi-Yau variety, and $w_0:Y_0\rightarrow\mathbb{K}$ is the superpotential determined by the enumeration of Maslov index 2 holomorphic discs bounded by the fibers of a (piecewise smooth) Lagrangian torus fibration on an open dense subset of $M$. See $\cite{aak}$, Theorem 1.5. The Berglund-H\"{u}bsch mirror of $M$, on the other hand, is 2-dimensional higher than the original Milnor fiber. This phenomenon also appears in $\cite{aak}$, Section 7.
\end{remark}

\section{Hochschild cohomology}\label{section:Hochschild}

This is supposed to be the most technical section of this paper, where we compute the Hochschild cohomology of the dg category of $\Gamma_w$-equivariant matrix factorizations $\mathit{MF}(\mathbb{K}^{n+1},\Gamma_w,w)$. We will keep the notations as in Section \ref{section:MF}. Although we have proved homological mirror symmetry for the dg category $\mathit{MF}(\mathbb{K}^{n+1},\Gamma_w,w+z_0\cdots z_{n+1})$ in order to keep the statement to be consistent with $\cite{lu1}$, Conjecture 1.4, for the purpose of doing computations we will use the (quasi-equivalent) dg category $\mathit{MF}(\mathbb{K}^{n+1},\Gamma_w,w)$ instead. 

Denote by $V$ the vector space over $\mathbb{K}$ spanned by $\{z_0,\cdots,z_{n+1}\}$. For $\gamma\in\Gamma$, let $V_\gamma\subset V$ be the $\gamma$-invariant subspace, and let $S_\gamma:=\mathit{Sym}(V_\gamma)$ be the corresponding symmetric algebra. For $w\in S$, denote by $w_\gamma$ the restriction of $w$ to $\mathrm{Spec}(S_\gamma)$. Let $\Gamma$ be an abelian finite extension of $\mathbb{G}_m$ acting linearly on $\mathbb{K}^{n+1}$, denote by $N_\gamma$ the complement of $V_\gamma\subset V$ so that $V\cong V_\gamma\oplus N_\gamma$ as a $\Gamma$-module.

The following result have appeared a couple of times in the literature, and the author learned it from $\cite{lu1}$.

\begin{theorem}[$\cite{bfk,ct,td,es}$]
Let $w\in\mathbb{K}[z_0,\cdots,z_{n+1}]$ be a non-zero element of degree $\chi\in\widecheck{\Gamma}$. Assume that the singular locus of the zero set $Z_{-w\boxplus w}$ of the Sebastiani–Thom sum $-w\boxplus w$ is contained in the product of the zero sets $Z_w\times Z_w$. Then $\mathit{HH}^k\left(\mathit{MF}(\mathbb{K}^{n+1},\Gamma,w)\right)$ is isomorphic to
\begin{equation}\label{eq:HH}
	\begin{split}
		&\bigoplus_{\gamma\in\ker\chi,l\geq0,k-\dim N_\gamma=2u}\left(H^{-2l}(dw_\gamma)\otimes\wedge^{\dim N_\gamma}N_\gamma^\vee\right)_{(u+l)\chi}\oplus \\
		&\bigoplus_{\gamma\in\ker\chi,l\geq0,k-\dim N_\gamma=2u+1}\left(H^{-2l-1}(dw_\gamma)\otimes\wedge^{\dim N_\gamma}N_\gamma^\vee\right)_{(u+l)\chi},
	\end{split}
\end{equation}
where $H^i(dw_\gamma)$ is the cohomology of the Koszul complex
\begin{equation}\label{eq:cpx}
C^\ast(dw_\gamma):=\left\{\cdots\rightarrow\wedge^2V_\gamma^\vee\otimes S_\gamma(-2\chi)\rightarrow V_\gamma^\vee\otimes S_\gamma(-\chi)\rightarrow S_\gamma\right\},
\end{equation}
where the rightmost term $S_\gamma$ sits in cohomological degree 0, and the differential is the contraction with
\begin{equation}
dw_\gamma\in(V_\gamma\otimes S_\gamma)_\chi.
\end{equation}
\end{theorem}

Let us describe a few concrete situations that are relevant to our computations, in which case the expression (\ref{eq:HH}) gets simplified drastically. When $w_\gamma$ has an isolated critical point at the origin, the cohomology of the Koszul complex (\ref{eq:cpx}) is concentrated in degree zero, and its zeroth cohomology is isomorphic to the Jacobi ring $\mathit{Jac}_{w_\gamma}$. As a consequence, only the summand corresponding to $l=0$ in (\ref{eq:HH}), namely
\begin{equation}
\left(\mathit{Jac}_{w_\gamma}\otimes\wedge^{\dim N_\gamma}N_\gamma^\vee\right)_{u\chi},
\end{equation}
would contribute to $\mathit{HH}^k\left(\mathit{MF}(\mathbb{K}^{n+1},\Gamma,w)\right)$, where $k=2u+\dim N_\gamma$.

Moreover, if $V_\gamma$ contains $\mathbb{K}z_0$, then $C^\ast(dw_\gamma)$ is isomorphic to the tensor product
\begin{equation}
C^\ast(dw_\gamma')\otimes\left\{\mathbb{K}z_0^\vee\otimes\chi^\vee\otimes\mathbb{K}[z_0]\rightarrow\mathbb{K}[z_0]\right\},
\end{equation}
where $w_\gamma'$ is the restriction of $w$ to the complement of $\mathbb{K}z_0\subset V_\gamma$, and the latter complex is concentrated in cohomological degrees $\left\{-1,0\right\}$ with vanishing differential. If $w_\gamma'$ has an isolated critical point at the origin, then as above the Koszul complex $C^\ast(dw_\gamma')$ is isomorphic to $\mathit{Jac}_{w_\gamma'}$, which is concentrated in degree 0, so only the summands
\begin{equation}
\left(\mathit{Jac}_{w_\gamma'}\otimes\mathbb{K}[z_0]\otimes\wedge^{\dim N_\gamma}N_\gamma^\vee\right)_{u\chi}
\end{equation}
and
\begin{equation}
\left(\mathbb{K}z_0^\vee\otimes\mathit{Jac}_{w_\gamma'}\otimes\mathbb{K}[z_0]\otimes\wedge^{\dim N_\gamma}N_\gamma^\vee\right)_{u\chi},
\end{equation}
which correspond to $l=0$ in (\ref{eq:HH}), would contribute to $\mathit{HH}^k\left(\mathit{MF}(\mathbb{K}^{n+1},\Gamma,w)\right)$, where $k=2u+\dim N_\gamma$ and $2u+\dim N_\gamma+1$ respectively.
\bigskip

As we have remarked in the introduction, what is relevant to our application is the Hochschild cohomology group $\mathit{HH}^0\left(\mathit{MF}(\mathbb{K}^{n+2},\Gamma_w,w)\right)$, and we now show that it is finite dimensional. On the other hand, as the reader will see, $\mathit{HH}^n\left(\mathit{MF}(\mathbb{K}^{n+2},\Gamma_w,w)\right)$ can also be computed explicitly in our situation, which is a fact that will be explored in Section \ref{section:argument}. For the purpose of proving Theorem \ref{theorem:main}, we will not attempt to do computations in the general case when $w$ is a polynomial of the form (\ref{eq:conic}). Instead, simplifying assumptions will be imposed to minimize our computational burden.

\begin{proposition}\label{proposition:HH}
Let $w$ be a polynomial as in (\ref{eq:conic}), regarded as an element of the ring $\mathbb{K}[z_0,\cdots,z_{n+1}]$, and assume that the powers $k_3,\cdots,k_{n+1}$ in the Brieskorn-Pham polynomial $p(z_{n+3},\cdots,z_{n+1})$ are distinct prime numbers, with $3\leq k_3\leq\cdots\leq k_{n+1}$. Then we have
\begin{equation}\label{eq:0}
\dim_\mathbb{K}\mathit{HH}^0\left(\mathit{MF}(\mathbb{K}^{n+2},\Gamma_w,w)\right)=k_3-1,
\end{equation}
and
\begin{equation}\label{eq:n}
\dim_\mathbb{K}\mathit{HH}^n\left(\mathit{MF}(\mathbb{K}^{n+2},\Gamma_w,w)\right)=\mu,
\end{equation}
where $\mu=(k_3-1)\cdots(k_{n+1}-1)$ is the Milnor number of the affine hypersurface $M$ defined by the polynomial $\{w=1\}\subset\mathbb{C}^{n+1}$.
\end{proposition}
\begin{proof}
By definition,
\begin{equation}\label{eq:Gamma}
\Gamma_w=\left\{(t_0,\cdots,t_{n+1})\in(\mathbb{G}_m)^{n+2}|t_1^2=t_2^2=t_3^{k_3}=\cdots=t_{n+1}^{k_{n+1}}=t_0\cdots t_{n+1}\right\},
\end{equation}
which implies that
\begin{equation}
\ker(\chi)\cong(\mathbb{Z}/2)^2\times\mathbb{Z}/k_3\times\cdots\mathbb{Z}/k_n.
\end{equation}
The character group $\widecheck{\Gamma}$ is generated by $\chi$ and $\chi_i=\deg(z_i)$ for $0\leq i\leq n+1$ with relations
\begin{equation}\label{eq:relation}
\chi=2\chi_1=2\chi_2=k_3\chi_3=\cdots=k_{n+1}\chi_{n+1}=\chi_0+\cdots+\chi_{n+1}.
\end{equation}
For any $\gamma\in\ker(\chi)$, one has
\begin{equation}
	\begin{split}
		\mathit{Jac}_{w_\gamma}\cong&\left\{\begin{array}{ll}
			\mathbb{K}[z_0] & \textrm{if }\mathbb{K}z_0\subset V_\gamma \\ \mathbb{K} & \textrm{if }\mathbb{K}z_0\not\subset V_\gamma
		\end{array}\right.
		\otimes
		\left\{\begin{array}{ll}
			\mathbb{K}[z_3]/(z_3^{k_3-1}) & \textrm{if }\mathbb{K}z_3\subset V_\gamma \\ \mathbb{K} & \textrm{if }\mathbb{K}z_3\not\subset V_\gamma
		\end{array}\right. \\
	&\otimes\cdots\otimes
	\left\{\begin{array}{ll}
		\mathbb{K}[z_{n+1}]/(z_{n+1}^{k_{n+1}-1}) & \textrm{if }\mathbb{K}z_{n+1}\subset V_\gamma \\ \mathbb{K} & \textrm{if }\mathbb{K}z_{n+1}\not\subset V_\gamma
	\end{array}\right.
	\end{split}
\end{equation}
Note that by our assumptions, none of the truncated polynomial rings $\mathbb{K}[z_i]/(z_i^{k_i-1})$ for $3\leq i\leq n+1$ could be isomorphic to $\mathbb{K}$. Write an element of $\mathit{Jac}_{w_\gamma}\otimes\wedge^{\dim N_\gamma}N_\gamma^\vee$ as
\begin{equation}
z_0^{a_0}z_3^{a_3}\cdots z_{n+1}^{a_{n+1}}\otimes z_{j_1}^\vee\wedge\cdots\wedge z_{j_s}^\vee,
\end{equation}
where $a_i=0$ if $\mathbb{K}z_i\not\subset V_\gamma$ for $i=0,3,\cdots,n+1$, and $a_i\leq k_i-2$ for $i=3,\cdots,n+1$. Its degree is
\begin{equation}\label{eq:deg}
a_0\chi_0+a_3\chi_3+\cdots+a_{n+1}\chi_{n+1}-\chi_{j_1}-\cdots-\chi_{j_s}.
\end{equation}
By (\ref{eq:relation}), (\ref{eq:deg}) can be proportional to $\chi$ if $V_\gamma\cap(\mathbb{K}z_1\oplus\mathbb{K}z_2)$ is either $\mathbb{K}z_1\oplus\mathbb{K}z_2$ or $\emptyset$.

We first look at the case when $V_\gamma=\emptyset$, where one must have
\begin{equation}
t_1=t_2=-1,t_3\in(\mathbb{Z}/k_3)\setminus\{1\},\cdots,t_{n+1}\in(\mathbb{Z}/{k_{n+1}})\setminus\{1\},
\end{equation}
since $t_0=t_3^{-1}\cdots t_{n+1}^{-1}\neq1$ by our assumptions that $k_3,\cdots,k_{n+1}$ are distinct prime numbers. There are $\mu$ such $\gamma$, and each one of them contributes a copy of $\mathbb{K}$ to the Hochschild cohomology group $\mathit{HH}^n\left(\mathit{MF}(\mathbb{K}^{n+2},\Gamma_w,w)\right)$, since the degree of
\begin{equation}
z_0^\vee\wedge z_1^\vee\wedge\cdots\wedge z_{n+1}^\vee\in\wedge^{n+2}V^\vee
\end{equation}
is $-\chi$.

To prove the identity (\ref{eq:n}), it remains to show that the other possibilities of $V_\gamma$ do not contribute to $\mathit{HH}^n\left(\mathit{MF}(\mathbb{K}^{n+2},\Gamma_w,w)\right)$. To do this, we separate our discussions into two cases, namely
\begin{itemize}
	\item[(i)] $\mathbb{K}z_0\subset V_\gamma$,
	\item[(ii)] $\mathbb{K}z_0\not\subset V_\gamma$ and $V_\gamma\neq\emptyset$.
\end{itemize}

In the first case, since $t_1^2=t_2^2=1$, it follows from (\ref{eq:Gamma}) that $t_i\in\mathbb{Z}/k_i$ for $i=3,\cdots,n+1$ and $t_3\cdots t_{n+1}=1$. By our assumption that $k_3,\cdots,k_{n+1}$ are distinct primes, the latter is possible only when $t_3=\cdots=t_{n+1}=1$. Thus $\mathbb{K}z_0\subset V_\gamma$ if and only if
\begin{itemize}
	\item $\gamma=(1,\cdots,1)$, in which case $V_\gamma=V$;
	\item $\gamma=(1,-1,-1,1,\cdots,1)$, in which case $V_\gamma=\mathbb{K}z_0\oplus\mathbb{K}z_3\oplus\cdots\oplus\mathbb{K}z_{n+1}$.
\end{itemize}

When $V_\gamma=V$, $w_\gamma=w$, the degree of $z_0^{a_0}z_3^{a_3}\cdots z_{n+1}^{a_{n+1}}$ is
\begin{equation}
a_0\chi-a_0\chi_1-a_0\chi_2-(a_0-a_3)\chi_3-\cdots-(a_0-a_{n+1})\chi_{n+1},
\end{equation}
where (\ref{eq:relation}) was used in the derivation of the above formula. This is proportional to $\chi$ if and only if
\begin{equation}
2|a_0,k_3|(a_0-a_3),\cdots,k_{n+1}|(a_0-a_{n+1}).
\end{equation}
Write $a_0=2r$, where $r\in\mathbb{Z}_{\geq0}$ and $a_i=2r-k_im$, where $m$ is an integer satisfying
\begin{equation}\label{eq:m}
\frac{2r+2}{k_i}-1\leq m\leq\frac{2r}{k_i}
\end{equation}
for $i=3,\cdots,n+1$, we have
\begin{equation}
\deg\left(z_0^{2r}z_3^{2r-k_3m}\cdots z_{n+1}^{2r-k_{n+1}m}\right)=m(1-n)\chi.
\end{equation}
It follows that each element of the form
\begin{equation}
z_0^{2r}z_3^{2r-k_3m}\cdots z_{n+1}^{2r-k_{n+1}m}\in\mathit{Jac}_w\otimes\mathbb{K}[z_0]
\end{equation}
contributes a copy of $\mathbb{K}$ to $\mathit{HH}^{2m(1-n)}\left(\mathit{MF}(\mathbb{K}^{n+2},\Gamma_w,w)\right)$. Similarly, each element of the form
\begin{equation}
z_0^\vee\otimes z_0^{2r}z_3^{2r-k_3m}\cdots z_{n+1}^{2r-k_{n+1}m}\in\mathbb{K}z_0^\vee\otimes\mathit{Jac}_w\otimes\mathbb{K}[z_0]
\end{equation}
contributes a copy of $\mathbb{K}$ to $\mathit{HH}^{2m(1-n)+1}\left(\mathit{MF}(\mathbb{K}^{n+2},\Gamma_w,w)\right)$. It is easy to find out that
\begin{equation}
2m(1-n)\neq n,n-1
\end{equation}
for any $m\in\mathbb{Z}$ whenever $n\geq3$, and the only integer solution for the equation $2m(1-n)=n$ is $m=-1$ and $n=2$. To exclude this possibility, we make use of the condition (\ref{eq:m}), which implies in this case that $\frac{2r+2}{k_3}\leq0$, and therefore contradicts with the fact that $r\in\mathbb{Z}_{\geq0}$. We have proved that when $V_\gamma=V$, there is no additional contributions to $\mathit{HH}^n\left(\mathit{MF}(\mathbb{K}^{n+2},\Gamma_w,w)\right)$.

When $V_\gamma=\mathbb{K}z_0\oplus\mathbb{K}z_3\oplus\cdots\oplus\mathbb{K}z_{n+1}$, using (\ref{eq:relation}), we can compute the degree of
\begin{equation}
z_0^{a_0}z_3^{a_3}\cdots z_{n+1}^{a_{n+1}}\otimes z_1^\vee\wedge z_2^\vee\in\mathit{Jac}_{w_\gamma}\otimes\wedge^2N_\gamma^\vee,
\end{equation}
which is
\begin{equation}
a_0\chi-(a_0+1)\chi_1-(a_0+1)\chi_2-(a_0-a_3)\chi_3-\cdots-(a_0-a_{n+1})\chi_{n+1}.
\end{equation}
This is proportional to $\chi$ is and only if
\begin{equation}
2|a_0+1,k_3|(a_0-a_3),\cdots,k_{n+1}|(a_0-a_{n+1}).
\end{equation}
Write $a_0=2r-1$ with $r\in\mathbb{N}$, and $a_i=2r-1-k_im$, where $m\in\mathbb{Z}$ satisfies
\begin{equation}\label{eq:m1}
\frac{2r+1}{k_i}-1\leq m\leq\frac{2r-1}{k_i}
\end{equation}
for $i=3,\cdots,n+1$. We have
\begin{equation}
\deg\left(z_0^{2r-1}z_3^{2r-1-k_3m}\cdots z_{n+1}^{2r-1-k_{n+1}m}\otimes z_1^\vee\wedge z_2^\vee\right)=-\left((n-1)m+1\right)\chi.
\end{equation}
It follows that each element of the form
\begin{equation}
z_0^{2r-1}z_3^{2r-1-k_3m}\cdots z_{n+1}^{2r-1-k_{n+1}m}\otimes z_1^\vee\wedge z_2^\vee\in\mathit{Jac}_{w_\gamma'}\otimes\mathbb{K}[z_0]\otimes\wedge^2N_\gamma^\vee
\end{equation}
contributes a copy of $\mathbb{K}$ to $\mathit{HH}^{2m(1-n)}\left(\mathit{MF}(\mathbb{K}^{n+2},\Gamma_w,w)\right)$ . Similarly, each element of the form
\begin{equation}
z_0^\vee\otimes z_0^{2r-1}z_3^{2r-1-k_3m}\cdots z_{n+1}^{2r-1-k_{n+1}m}\otimes z_1^\vee\wedge z_2^\vee\in\mathbb{K}z_0^\vee\otimes\mathit{Jac}_{w_\gamma'}\otimes\mathbb{K}[z_0]\otimes\wedge^2N_\gamma^\vee
\end{equation}
contributes a copy of $\mathbb{K}$ to $\mathit{HH}^{2m(1-n)+1}\left(\mathit{MF}(\mathbb{K}^{n+2},\Gamma_w,w)\right)$. It is easy to see that if $n\geq3$, then the equations
\begin{equation}
2m(1-n)=n-1 \textrm{ and }2m(1-n)=n
\end{equation}
do not admit any integer solutions, and when $n=2$ and $m=-1$, it would contradict (\ref{eq:m1}), as we have seen above. In this way, we have proved that when $V_\gamma=\mathbb{K}z_0\oplus\mathbb{K}z_3\oplus\cdots\oplus\mathbb{K}z_{n+1}$, there is no additional contributions to $\mathit{HH}^n\left(\mathit{MF}(\mathbb{K}^{n+2},\Gamma_w,w)\right)$.

This completes the discussion for case (i). For case (ii), we first assume that $V_\gamma\cap(\mathbb{K}z_1\oplus\mathbb{K}z_2)=\emptyset$. Since $V_\gamma\neq\emptyset$, there is a non-empty subset $\{j_1,\cdots,j_s\}\subset\{3,\cdots,n+1\}$, where $s\leq n-1$, such that
\begin{equation}
V_\gamma=\mathbb{K}z_{j_1}\oplus\cdots\oplus\mathbb{K}z_{j_s}.
\end{equation}
Without loss of generality, we may assume that $\{j_1,\cdots,j_s\}=\{3,\cdots,s+2\}$, so an element
\begin{equation}
z_3^{a_3}\cdots z_{s+2}^{a_{s+2}}\otimes z_0^\vee\wedge z_1^\vee\wedge z_2^\vee\wedge z_{s+3}^\vee\wedge\cdots\wedge z_{n+1}^\vee\in\mathit{Jac}_{w_\gamma}\otimes\wedge^{\dim N_\gamma}N_\gamma^\vee
\end{equation}
has degree
\begin{equation}
(a_3+1)\chi_3+\cdots+(a_{s+2}+1)\chi_{s+2}-\chi,
\end{equation}
which cannot be proportional to $\chi$ since $a_i+1<k_i$ for $3\leq i\leq n+1$.

If $\mathbb{K}z_1\oplus\mathbb{K}z_2\subset V_\gamma$, the subset $\{j_1,\cdots,j_s\}$ above may be empty. In the non-empty case, the degree of
\begin{equation}
z_3^{a_3}\cdots z_{s+2}^{a_{s+2}}\otimes z_0^\vee\wedge z_{s+3}^\vee\wedge\cdots\wedge z_{n+1}^\vee\in\mathit{Jac}_{w_\gamma}\otimes\wedge^{\dim N_\gamma}N_\gamma^\vee
\end{equation}
is given by
\begin{equation}
(a_3+1)\chi_3+\cdots+(a_{s+2}+1)\chi_{s+2},
\end{equation}
which cannot be proportional to $\chi$ for the same reason as above.

Finally, when $V_\gamma=\mathbb{K}z_1\oplus\mathbb{K}z_2$, the only possible contribution to $\mathit{HH}^\ast\left(\mathit{MF}(\mathbb{K}^{n+2},\Gamma_w,w)\right)$ comes from the element
\begin{equation}\label{eq:top}
z_0^\vee\wedge z_3^\vee\wedge\cdots\wedge z_{n+1}^\vee\in\wedge^nN_\gamma^\vee,
\end{equation}
whose degree is not proportional to $\chi$, so it cannot contribute.

We now prove (\ref{eq:0}). From the above analysis, one can see that the potential contributions to $\mathit{HH}^0\left(\mathit{MF}(\mathbb{K}^{n+2},\Gamma_w,w)\right)$ happen in the situations when $V_\gamma=V$ and $V_\gamma=\mathbb{K}z_0\oplus\mathbb{K}z_3\oplus\cdots\mathbb{K}z_{n+1}$.

In the former case, possible contributions to $\mathit{HH}^0\left(\mathit{MF}(\mathbb{K}^{n+2},\Gamma_w,w)\right)$ can occur only when $m(1-n)=0$, which happens when $m=0$. It follows from (\ref{eq:m}) that
\begin{equation}
2\leq 2r+2\leq k_3.
\end{equation}
Clearly, there can be $\left\lfloor\frac{k_3-2}{2}\right\rfloor+1$ such $r$, and each one corresponds to a generator of $\mathit{HH}^0\left(\mathit{MF}(\mathbb{K}^{n+2},\Gamma_w,w)\right)$.

In the latter case, one gets a contribution to $\mathit{HH}^0\left(\mathit{MF}(\mathbb{K}^{n+2},\Gamma_w,w)\right)$ when $m=0$. By (\ref{eq:m1}) we have
\begin{equation}
3\leq 2r+1\leq k_3.
\end{equation}
It is elementary to see that there are $\left\lfloor\frac{k_3-1}{2}\right\rfloor$ such $r$. All together, we get
\begin{equation}
\left\lfloor\frac{k_3-2}{2}\right\rfloor+\left\lfloor\frac{k_3-1}{2}\right\rfloor+1=k_3-1
\end{equation}
copies of $\mathbb{K}$ in $\mathit{HH}^0\left(\mathit{MF}(\mathbb{K}^{n+2},\Gamma_w,w)\right)$, which completes the proof of (\ref{eq:0}).
\end{proof}

\begin{remark}
Note that when $n=2$, our computation coincides with the computation of $\cite{etl1}$ for the symplectic cohomology of a 4-dimensional $A_{k_3-1}$ Milnor fiber.
\end{remark}

\section{The main argument}\label{section:argument}

This section is devoted to the proof of our main result. We start by recalling the following algebraic fact:

\begin{theorem}[\"{O}inert]\label{theorem:unit}
Let $\mathbb{K}$ be any field, and $G$ a torsion-free group, then any central unit in the group algebra $\mathbb{K}[G]$ has the form $\lambda\cdot g$, with $\lambda\in\mathbb{K}^\times$ and $g\in G$.
\end{theorem}

\begin{remark}
This is a special case of Kaplansky's famous unit conjecture, which states that any unit in $\mathbb{K}[G]$ of a torsion-free group $G$ must be of the form $\lambda\cdot g$. When $\mathrm{char}(\mathbb{K})=2$, this conjecture has been disproved recently by Gardam $\cite{gg}$.
\end{remark}

The above theorem follows from $\cite{jo}$, Theorem 6.2, which deals with the more general situation of unital rings equipped with non-degenerate $G$-gradings. For readers' convenience, we will include a self-contained proof of Theorem \ref{theorem:unit} in Appendix \ref{section:A}.

\begin{lemma}\label{lemma:0}
Let $M$ be a Liouville manifold admitting a quasi-dilation. If $L\subset M$ is a closed, orientable, exact Lagrangian submanifold which is a $K(\pi,1)$ space, then $\mathit{SH}^0(M)$ must be infinite-dimensional.
\end{lemma}
\begin{proof}
By Theorem \ref{theorem:dilation}, the existence of an exact Lagrangian $K(\pi,1)$ prohibits the existence of a dilation in $\mathit{SH}^1(M)$. Since $M$ admits a quasi-dilation, it follows that there is a non-trivial unit $h\in\mathit{SH}^0(M)^\times$ such that the BV operator $\Delta:\mathit{SH}^1(M)\rightarrow\mathit{SH}^0(M)$ hits $h$. Recall that the cochain complex $\mathit{SC}^\ast(M)$ which computes the symplectic cohomology is defined by (a small perturbation of) a Hamiltonian function $H:M\rightarrow\mathbb{R}$ which is a $C^2$-small Morse function $\phi$ when restricted to the interior of $\overline{M}$, so it contains the Morse complex $\mathit{CM}^\ast(\phi)$ as a subcomplex. Denote by $\mathit{SH}^\ast_+(M)$ the cohomology of the quotient complex $\mathit{SC}^\ast_+:=\mathit{SC}^\ast(M)/\mathit{CM}^\ast(\phi)$ generated by the non-constant Hamiltonian orbits. One can write
\begin{equation}\label{eq:h}
h=\alpha\cdot1+t,
\end{equation}
with $\alpha\in\mathbb{K}$ and $t\in\mathit{SH}^0_+(M)$ is non-zero. We claim that $\alpha=0$.

Let $L\subset M$ be a closed, orientable, exact Lagrangian submanifold which is a $K(\pi,1)$ space, consider the Viterbo functoriality $\mathit{SH}^\ast(M)\rightarrow\mathit{SH}^\ast(T^\ast L)$, which fits into a commutative diagram
\begin{equation}
	\begin{tikzcd}
		&\mathit{SH}^{\ast}(M) \arrow[d,"{\Delta}"'] \arrow[r] &\mathit{SH}^\ast(T^\ast L) \arrow[d,"{\Delta}"] \\
		&\mathit{SH}^{\ast-1}(M) \arrow[r] &\mathit{SH}^{\ast-1}(T^\ast L)
	\end{tikzcd}
\end{equation}
under which the element $h\in\mathit{SH}^0(M)^\times$ goes to a class
\begin{equation}
h_L=\alpha\cdot 1_L+t_L\in\mathit{SH}^0(T^\ast L)^\times
\end{equation}
which lies in the image of the BV operator, where $1_L$ denotes the identity and $t_L\in\mathit{SH}^0_+(T^\ast L)$ is non-zero. Fix an orientation of $L$, and a background class $\nu\in H^2(T^\ast L;\mathbb{Z}/2)$ given by the pullback of the second Stiefel-Whitney class $w_2(L)$, we have the isomorphisms
\begin{equation}\label{eq:iso}
\mathit{SH}^0(T^\ast L)\cong\mathit{HH}^0(\mathcal{W}(T^\ast L))\cong\mathit{HH}^0(C_{-\ast}(\Omega_pL))\cong Z(\mathbb{K}[\pi_1(L)]),
\end{equation}
where the symplectic cohomology and the wrapped Fukaya category are defined with respect to $\nu$, $C_{-\ast}(\Omega_pL)$ is the dg algebra of chains on the based loop space, and $Z(\mathbb{K}[\pi_1(L)])$ is the center of the fundamental group algebra. The first isomorphism is given by the closed-open string map $\mathit{CO}$ $\cite{sg2}$, the second isomorphism follows from the generation result of Abouzaid $\cite{ma1}$, and both of the first two isomorphisms rely on the assumption that $L$ is orientable, while the third isomorphism follows from the assumption that $L$ is a $K(\pi,1)$ space. Under the composition of these isomorphisms, the class $h_L$ can be identified with a non-trivial\footnote{The reader may find the terminology here confusing: in this paper, we consider only multiples of the the identity as trivial units in the group algebra $\mathbb{K}[G]$, while for group theorists, any $\lambda\cdot g$ for $g\in G$ and $\lambda\in\mathbb{K}^\times$ is trivial.} central unit in $\mathbb{K}[\pi_1(L)]$. Since $L$ is a $K(\pi,1)$ space, $\pi_1(L)$ is torsion-free. It follows from Theorem \ref{theorem:unit} that the image of $h_L$ must have vanishing constant coefficient, which forces $\alpha=0$. We have concluded that $h=t\in\mathit{SH}_+^0(M)$.

We proceed to show that the classes $1,h,h^2,\cdots\in\mathit{SH}^0(M)$ are linearly independent, which would in particular imply that $\mathit{SH}^0(M)$ is infinite-dimensional. First notice that for any $m\in\mathbb{N}$, $h^m\in\mathit{SH}_+^0(M)$. To see this we again use the Viterbo functoriality $\mathit{SH}^\ast(M)\rightarrow\mathit{SH}^\ast(T^\ast L)$ and the isomorphism (\ref{eq:iso}), under which the elements $h^m\in\mathit{SH}^0(M)$ are sent to powers of a non-trivial central unit in $\mathbb{K}[\pi_1(L)]$. By Theorem \ref{theorem:unit}, these powers must have vanishing constant coefficients, which implies that $h^m\in\mathit{SH}_+^0(M)$. Suppose that there are scalars $c_0,\cdots,c_m\in\mathbb{K}$ such that
\begin{equation}
c_0\cdot1+c_1\cdot h+\cdots+c_m\cdot h^m=0.
\end{equation}
On the cochain level, one can find a cocycle $\eta\in\mathit{SC}_+^0(M)$ representing $h$ such that
\begin{equation}\label{eq:dependence}
c_0\cdot e+c_1\cdot\eta+\cdots+c_m\cdot\eta^m=d\beta
\end{equation}
for some cochain $\beta\in\mathit{SC}_+^{-1}(M)$, where $e\in\mathit{CM}^0(\phi)$ corresponds to the unique minimum of the Morse function $\phi$. Since all the terms except for the first one on the left-hand side of (\ref{eq:dependence}) lie in the complement of the subspace $\mathit{CM}^0(\phi)\subset\mathit{SC}^0(M)$, it follows that $c_0=0$, and we are reduced to dealing with the equation
\begin{equation}
c_1\cdot 1+c_2\cdot h+\cdots+c_m\cdot h^{m-1}=0
\end{equation}
in $\mathit{SH}^0(M)$. Arguing by induction shows that $c_0=c_1=\cdots=c_m=0$.
\end{proof}

\begin{remark}\label{remark:0}
\begin{itemize}
	\item[(i)] One can prove the infinite-dimensionality of $\mathit{SH}^0(M)$ under the weaker assumption that $M$ admits a cyclic dilation. For the purpose of this paper, introducing this additional sophistication is uncalled for, but this has potential applications in proving the nonexistence of exact Lagrangian $K(\pi,1)$ spaces in more general smooth affine varieties with log Kodaira dimension $-\infty$, see Section \ref{section:cyclic}.
	\item[(ii)] Similar considerations have appeared in the literature with the invertible element $h\in\mathit{SH}^0(M)^\times$ with zero constant coefficient being replaced by the \textit{Borman-Sheridan class} $s\in\mathit{SH}^0(M)$. It is argued in $\cite{dt}$, Proposition 3.8 that if $M$ is a Liouville subdomain in the complement of a Donaldson hypersurface in a compact monotone symplectic manifold, and $L\subset M$ is an exact Lagrangian torus whose Landau-Ginzburg superpotential $w_L$ is a non-constant Laurent polynomial, then the powers $s^k\in\mathit{SH}^0(M)$ must be linearly independent. In particular, $\mathit{SH}^0(M)$ is infinite-dimensional.
\end{itemize}
\end{remark}

\begin{example}
When $M=T^\ast T^n$ is the cotangent bundle of an $n$-dimensional torus, we have an isomorphism
\begin{equation}\label{eq:torus}
\mathit{SH}^\ast(M)\cong\mathbb{K}\left[H^1(T^n;\mathbb{Z})\right]\otimes H^\ast(T^n;\mathbb{K})
\end{equation}
as $\mathbb{K}$-algebras, which in particular shows that $\mathit{SH}^0(M)$ is infinite-dimensional. It is well-known that $M$ admits a quasi-dilation. More sophisticated examples of Liouville manifolds satisfying the assumptions of Lemma \ref{lemma:0} are given by the affine conic bundles (\ref{eq:GP}) over $(\mathbb{C}^\ast)^{n-1}$ in Example \ref{example:quasi-dilation}. The infinite-dimensionality of their $\mathit{SH}^0(M)$ can be seen from the existence of a distinguished class $\mathring{s}\in\mathit{SH}^0(M)$ which is mirror to a non-vanishing regular function $\mathring{w}_0:Y_0\rightarrow\mathbb{K}^\times$, whose ``quantum correction" gives the superpotential $w_0:Y_0\rightarrow\mathbb{K}$ mentioned in Remark \ref{remark:LG}. See $\cite{aak}$, Section 6.
\end{example}

\begin{proof}[Proof of Theorem \ref{theorem:main}]
Let $M$ be a Milnor fiber as in the assumption of Theorem \ref{theorem:main}, and let $L\subset M$ be a closed, orientable exact Lagrangian submanifold which is a $K(\pi,1)$ space. As we have explained in the introduction, since $M$ is the Milnor fiber of a singularity which is doubly stabilized, it admits a quasi-dilation. Lemma \ref{lemma:0} can therefore be applied to $M$ to conclude that $\mathit{SH}^0(M)$ is infinite-dimensional, with an invertible element $h\in\mathit{SH}_+^0(M)$ whose powers span an infinite-dimensional subspace.

On the other hand, Theorem \ref{theorem:HMS} proved in Section \ref{section:MF}, combined with the fact that the closed-open string map (\ref{eq:CO}) is an isomorphism, gives rise to the isomorphism
\begin{equation}
\mathit{SH}^0(M)\cong\mathit{HH}^0\left(\mathit{MF}(\mathbb{K}^{n+2},\Gamma_w,w)\right).
\end{equation}
Let us temporarily restrict ourselves to the special case when the powers $k_3,\cdots,k_{n+1}$ in the Brieskorn-Pham polynomial $p$ are distinct prime numbers. By Proposition \ref{proposition:HH}, it follows that for such $M$ one has
\begin{equation}\label{eq:dim}
\dim_\mathbb{K}\mathit{SH}^0(M)<\infty,
\end{equation}
which contradicts with Lemma \ref{lemma:0}.

We have proved the nonexistence of exact Lagrangian $K(\pi,1)$ spaces in $M$ in the special case when $p$ is a Brieskorn-Pham polynomial whose powers are distinct prime numbers. Since for any Milnor fiber $M$ of a singularity of the form $\{w=0\}\subset\mathbb{C}^{n+1}$, where $w$ is as in (\ref{eq:conic}), there is a Milnor fiber $M'$ associated to the singularity
\begin{equation}
z_1^2+z_2^2+z_3^{p_3}+\cdots+z_{n+1}^{p_{n+1}}=0,
\end{equation}
where $p_3,\cdots,p_{n+1}$ are distinct prime numbers satisfying
\begin{equation}
p_3\geq k_3,\cdots,p_{n+1}\geq k_{n+1},
\end{equation}
the general case follows from a standard degeneration argument, which implies the existence of a Liouville embedding $\overline{M}\hookrightarrow M'$.
\end{proof}

We now use the computation of $\mathit{HH}^n\left(\mathit{MF}(\mathbb{K}^{n+2},\Gamma_w,w)\right)$ in Proposition \ref{proposition:HH} and the Koszul duality between Fukaya $A_\infty$-algebras proved in Proposition \ref{proposition:Koszul} to give a proof to Proposition \ref{proposition:PSS}.

\begin{proof}[Proof of Proposition \ref{proposition:PSS}]
By $\cite{ol}$, Theorem 1.8, there is a commutative diagram
\begin{equation}\label{eq:Lazarev}
	\begin{tikzcd}
		&H^n(M;\mathbb{K}) \arrow[d,"{\mathscr{L}}"'] \arrow[r,"{\mathit{PSS}}"] &\mathit{SH}^n(M)  \\
		&K_0(\mathcal{W}(M)) \arrow[r,"{\mathit{tr}}"] &\mathit{HH}_0(\mathcal{W}(M)) \arrow[u,"{\mathit{OC}}"']
	\end{tikzcd}
\end{equation}
for any Weinstein manifold $M$, where $\mathscr{L}$ is a surjective map constructed by Lazarev ($\cite{ol}$, Theorem 1.4), and $\mathit{tr}$ is the Dennis trace map.
	
Let $M$ be as in the assumption of Proposition \ref{proposition:PSS}. Since the open-closed string map $\mathit{OC}$ is known to be an isomorphism $\cite{sg2}$, and our computations of $\mathit{HH}^n\left(\mathit{MF}(\mathbb{K}^{n+2},\Gamma_w,w)\right)$ in Proposition \ref{proposition:HH} shows that $H^n(M;\mathbb{K})$ and $\mathit{SH}^n(M)$ have the same rank, it suffices to prove that the Dennis trace map $\mathit{tr}:K_0(\mathcal{W}(M))\rightarrow\mathit{HH}_0(\mathcal{W}(M))$ is surjective. In fact, according to the generation result of $\cite{cdgg}$ and $\cite{gps2}$ mentioned in Section \ref{section:Koszul}, the Grothendieck group $K_0(\mathcal{W}(M))$ is freely generated by the classes $[L_1],\cdots,[L_\mu]$ of the cocores constructed in Section \ref{section:Koszul} using the Lefschetz presentation of $M$. Under $\mathit{tr}$, they are mapped to the Hochschild homology classes defined by the identity endomorphisms $e_{L_1},\cdots,e_{L_\mu}$ of the objects $L_1,\cdots,L_\mu$ which generate the wrapped Fukaya category $\mathcal{W}(M)$. Using the Koszul duality between the Fukaya $A_\infty$-algebras $\mathcal{F}_M$ and $\mathcal{W}_M$ proved in Proposition \ref{proposition:Koszul}, and the Morita invariance of Hochschild homology, we get an isomorphism
\begin{equation}\label{eq:duality}
\mathit{HH}_0(\mathcal{W}(M))\cong\hom_\mathbb{K}\left(\mathit{HH}_0(\mathcal{F}(M)),\mathbb{K}\right),
\end{equation}
under which the Hochschild homology class $[e_{L_i}]$, where $1\leq i\leq\mu$, is mapped to the dual of $[e_{V_i}]$, where as in Section \ref{section:Koszul}, $V_i\subset M$ is a vanishing cycle which intersects $L_i$ transversely at a unique point. Now (\ref{eq:n}), together with the fact that $\mathit{OC}$ is a isomorphism implies that
\begin{equation}
\dim_\mathbb{K}\mathit{HH}_0(\mathcal{F}(M))\cong\dim_\mathbb{K}\mathit{HH}_0(\mathcal{W}(M))\cong\dim_\mathbb{K}\mathit{SH}^n(M)=\mu,
\end{equation}
so $\mathit{HH}_0(\mathcal{F}(M))$ is freely generated by the orthogonal classes $[e_{V_1}],\cdots,[e_{V_\mu}]$. This allows us to conclude from (\ref{eq:duality}) that $\mathit{HH}_0(\mathcal{W}(M))$ is freely generated by the classes $[e_{L_1}],\cdots,[e_{L_\mu}]$.
\end{proof}

\begin{remark}\label{remark:PSS}
We expect the isomorphism $\mathit{PSS}:H^n(M;\mathbb{K})\xrightarrow{\cong}\mathit{SH}^n(M)$ to hold for Milnor fibers which are smooth affine varieties with log Kodaira dimension $-\infty$. Whether this is true seems to be interesting in its own right. In the case when $M$ is a plumbing of $T^\ast S^n$'s, where $n\geq3$, according to a tree, this fact is essentially proved in $\cite{ps5}$, Lecture 15, where it is shown to be useful in understanding the homology classes of odd-dimensional Lagrangian spheres.
\end{remark}

Proposition \ref{proposition:PSS} is interesting also because of the following fact:

\begin{proposition}\label{proposition:n}
Let $M$ be a Liouville manifold which admits a cyclic dilation, and the PSS map $H^n(M;\mathbb{K})\rightarrow\mathit{SH}^n(M)$ is an isomorphism, then $M$ does not contain an exact Lagrangian torus with non-zero homology class $[L]\in H_n(M;\mathbb{K})$.
\end{proposition}
\begin{proof}
Denote by $v\in\mathit{SH}^n(M)$ the class coming from $[L]$ via the PSS map, and consider the non-zero class $vh\in\mathit{SH}^n(M)$ obtained by taking the pair-of-pants product with the invertible element $h\in\mathit{SH}^0_+(M)$ established in Lemma \ref{lemma:0}. Under the Viterbo functoriality $\mathit{SH}^\ast(M)\rightarrow\mathit{SH}^n(T^\ast L)$, $vh$ goes to a class of the form $h_L\otimes\mathit{PD}([\mathit{pt}])$, where we have used the identification (\ref{eq:torus}). As before, $h_L$ is a non-trivial (central) unit in the fundamental group algebra, and $\mathit{PD}([\mathit{pt}])$ is the Poincar\'{e} dual of the class of a point. The PSS map for the Liouville manifold $T^\ast L$ is given by the inclusion of constant loops via the obvious identification $H^\ast(T^\ast L;\mathbb{K})\cong H^\ast(L;\mathbb{K})$, from which one sees that the class $h_L\otimes\mathit{PD}([\mathit{pt}])$ lies in $\mathit{SH}_+^n(T^\ast L)$, so must be the class $vh\in\mathit{SH}^n(M)$, by the compatibility of the Viterbo functoriality with the map $\mathit{SH}^n(M)\rightarrow\mathit{SH}_+^n(M)$. This contradicts with the fact that $\mathit{SH}_+^n(M)=0$.
\end{proof}

By Eliashberg's regular Lagrangian conjecture, every oriented, closed exact Lagrangian submanifold $L\subset M$ should be homologically essential, as long as $M$ is a Weinstein manifold. Assuming this, the $n$th degree PSS map being an isomorphism provides a geometric criterion for the nonexistence of exact Lagrangian tori. It would be interesting to know whether the assumption that $M$ admits a cyclic dilation can be removed in the above proposition.

\section{Generalizations}\label{section:cyclic}

We discuss several possible generalizations of Theorem \ref{theorem:main} in the final section. 
\bigskip

In deriving the version of homological mirror symmetry for the wrapped Fukaya category of $M$ as stated in Theorem \ref{theorem:HMS}, we have used the work of Futaki-Ueda $\cite{fu}$ on the homological mirror symmetry of Brieskorn-Pham polynomials. This is the main reason why we restricted ourselves to the special case where the polynomial $p$ in the definition of the affine conic bundle $M$ is Brieskorn-Pham. There are ongoing efforts by Gammage-Smith and Polishchuk-Varolgunes $\cite{pva}$ to prove the Berglund-H\"{u}bsch homological mirror symmetry for a general weighted homogeneous polynomial. Because of this one may expect to apply the same method of this paper to prove the nonexistence of exact Lagrangian $K(\pi,1)$ spaces in affine conic bundles (\ref{eq:double}), with $p$ being a general invertible weighted homogeneous polynomial with an isolated critical point at the origin. However, one should notice that the computation of the corresponding Hochschild cohomology group, $\mathit{HH}^0\left(\mathit{MF}(\mathbb{K}^{n+2},\Gamma_w,w)\right)$, would become more complicated in general than what we have presented in Proposition \ref{proposition:HH}.

In another direction, in view of Theorem \ref{theorem:dilation}, it seems to be legitimate to expect that the nonexistence of exact Lagrangian $K(\pi,1)$ spaces in affine conic bundles is due to some deeper reason:

\begin{conjecture}\label{conjecture:dilation}
Let $M$ be an affine conic bundle of the form (\ref{eq:double}), such that the polynomial $p(z_3,\cdots,z_{n+1})$ is weighted homogeneous and invertible. Then $M$ admits a dilation.
\end{conjecture}

There are also Milnor fibers beyond the affine conic bundles in Conjecture \ref{conjecture:dilation} which should admit dilations. For example, Let $M_{p,q,r}$ be the Milnor fiber associated to the isolated singularity
\begin{equation}\label{eq:pqr}
x^p+y^q+z^r+\lambda xyz+w^2=0,\textrm{ where }\frac{1}{p}+\frac{1}{q}+\frac{1}{r}\leq1,
\end{equation}
which are once stabilizations of the unimodal singularities studied in $\cite{ak1}$, where one can take $\lambda=1$ when $\frac{1}{p}+\frac{1}{q}+\frac{1}{r}<1$, and $\lambda=0$ otherwise. It is proved in $\cite{yl1}$, Section 4.2 that $M_{p,q,r}$ admits a quasi-dilation. In the case when $\frac{1}{p}+\frac{1}{q}+\frac{1}{r}=1$, the singularity (\ref{eq:pqr}) is weighted homogeneous, and one can apply similar arguments as used in this paper to prove the corresponding Berglund-H\"{u}bsch homological mirror symmetry for these polynomials and use the computational tools in Section \ref{section:Hochschild} to show that $\mathit{SH}^0(M_{p,q,r})$ is finite-dimensional. This implies the nonexistence of exact Lagrangian $K(\pi,1)$'s in the Milnor fibers $M_{3,3,3}$, $M_{4,4,2}$ and $M_{6,3,2}$, and it is therefore natural to expect that they all admit dilations. Of course, the nonexistence of exact Lagrangian $K(\pi,1)$'s in $M_{3,3,3}$ follows also from Theorem \ref{theorem:cyclic}, since it admits a Liouville embedding into the Milnor fiber of a 3-fold triple point. The same should hold for all the Milnor fibers $M_{p,q,r}$, even though the singularities (\ref{eq:pqr}) are in general not weighted homogeneous.

\begin{remark}
One powerful tool of proving the existence of a dilation is the Logarithmic PSS map introduced by Ganatra-Pomerleano $\cite{gp1,gp2}$. However, this usually relies on realizing $M$ as a Liouville subdomain of the complement of a normal crossing divisor in some smooth Fano variety, which is not possible for topological reasons when the degree of the polynomial $p$ in Conjecture \ref{conjecture:dilation} is sufficiently high.
\end{remark}

As we have mentioned in the introduction, the existence of a cyclic dilation $\tilde{b}\in\mathit{SH}_{S^1}^1(M)$ whose image under the marking map equals the identity is enough to deduce the nonexistence of exact Lagrangian $K(\pi,1)$'s in $M$. In $\cite{yl2}$, it is conjectured that this is the case for smooth affine varieties with log Kodaira dimension $-\infty$. Though this seems to be overoptimistic, the following should hold:

\begin{conjecture}
Let $M$ be a Milnor fiber of an isolated hypersurface singularity whose log Kodaira dimension is $-\infty$, then it admits a cyclic dilation with $h=1$.
\end{conjecture}

One may also use the existence of a cyclic dilation with arbitrary $h\in\mathit{SH}^0(M)^\times$ to exclude the existence of exact Lagrangian submanifolds which are $K(\pi,1)$ spaces. As a simplest class of examples, consider the Liouville manifolds $M_T$ obtained by plumbings of $T^\ast S^n$'s according to any tree $T$, where $n\geq3$. It is proved in $\cite{yl2}$, Proposition 6.2 that $M_T$ admits a cyclic dilation, possibly with $h\neq1$. In view of Remark \ref{remark:0} (ii), it follows that if $M_T$ contains an exact Lagrangian submanifold which is a $K(\pi,1)$ space, then $\mathit{SH}^0(M_T)$ is infinite-dimensional. On the other hand, by the commutative diagram (\ref{eq:Lazarev}), it follows that
\begin{equation}
\mathit{SH}^n(M_T)\cong H^n(M_T;\mathbb{K})\cong\mathbb{K}^{|T_0|},
\end{equation}
where $T_0$ denotes the set of vertices of $T$. Now Proposition \ref{proposition:n} suggests that $M$ does not contain an exact Lagrangian torus. The symplectic cohomology group $\mathit{SH}^0(M_T)$ should be computable as the Hochschild cohomology $\mathit{HH}^0(\mathcal{G}_T)$ of the Ginzburg dg algebra $\mathcal{G}_T$ associated to the tree $T$.

\appendix

\section{Central units in group algebras}\label{section:A}

We reproduce here the proof of $\cite{jo}$, Theorem 6.2 in the special case of a group algebra $\mathbb{K}[G]$, where $\mathbb{K}$ is any field, and $G$ is a torsion-free group. Any element $x\in\mathbb{K}[G]$ can be written as $x=\sum_{g\in G}c_gg$, where $c_g\in\mathbb{K}$, and only finitely many $c_g$ are non-zero. The \textit{support} of $x$ is defined to be
\begin{equation}
\mathrm{Supp}(x):=\{g\in G|c_g\neq0\}.
\end{equation}
$\mathrm{Supp}(x)$ generates a subgroup of $G$, which we will denote by $H_x$. To get a better understanding of the structure of $H_x\subset G$, we recall the following notion.

\begin{definition}
A group $G$ is called an \textit{FC-group} if each $g\in G$ has only a finite number of conjugates in $G$. In other words, $[G:C_G(g)]<\infty$, where $C_G(g)$ is the centralizer of $g$.
\end{definition}

As for a torsion-free group $G$, we have the following well-known result.

\begin{lemma}[Neumann $\cite{bn}$]\label{lemma:abel}
Every torsion-free FC-group is abelian.
\end{lemma}

The notion of an FC-group is closely related to our interest in central elements of $\mathbb{K}[G]$, as can be seen from the following fact.

\begin{proposition}\label{proposition:FC}
Let $G$ be any group, and $\mathbb{K}[G]$ be its group algebra. If $x\in\mathbb{K}[G]$ is a central element, then $H_x\subset G$ is an FC-group.
\end{proposition}
\begin{proof}
Define the subset
\begin{equation}
\Delta(G):=\{g\in G|g\textrm{ has only finitely many conjugates in }G\}.
\end{equation}
For any $a\in G$, consider the homogeneous element $y=c_aa\in\mathbb{K}[G]$ for some $c_a\in\mathbb{K}^\times$. Note that
\begin{equation}
a\mathrm{Supp}(x)=\mathrm{Supp}(yx)=\mathrm{Supp}(xy)=\mathrm{Supp}(x)a,
\end{equation}
which shows that $\mathrm{Supp}(x)$ is closed under conjugation by elements of $G$. Since $\mathrm{Supp}(x)$ is by definition a finite subset of $G$, it follows that $\mathrm{Supp}(x)\subset\Delta(G)$. Since $H_x$ is generated by $\mathrm{Supp}(x)$, we conclude that $H_x$ is an FC-group.
\end{proof}

We now turn our attention to units in the group algebra $\mathbb{K}[G]$. Given any subgroup $H\subset G$, there is a projection map $\pi_H:\mathbb{K}[G]\rightarrow\mathbb{K}[H]$, defined by
\begin{equation}
\pi_H\left(\sum_{g\in G}c_gg\right)=\sum_{g\in H}c_gg.
\end{equation}

\begin{lemma}\label{lemma:inv}
Let $H$ be a subgroup of $G$, and $x\in\mathbb{K}[H]$. Then $x$ is a left (right) invertible in $\mathbb{K}[H]$ if and only if it's left (right) invertible in $\mathbb{K}[G]$.
\end{lemma}
\begin{proof}
First note that for $x\in\mathbb{K}[G]$ and $y\in\mathbb{K}[H]$, we have $\pi_H(xy)=\pi_H(x)y$. In fact, put $x'=x-\pi_H(x)$, then $\mathrm{Supp}(x')\subset G\setminus H$. Note that if $g\in G\setminus H$ and $h\in H$, then $gh\notin H$, which implies that $\mathrm{Supp}(x'y)\subset G\setminus H$. Hence,
\begin{equation}\label{eq:prj}
\pi_H(xy)=\pi_H\left((x'+\pi_H(x))y\right)=\pi_H(x'y)+\pi_H(\pi_H(x)y)=\pi_H(x)y.
\end{equation}
Now suppose that $x\in\mathbb{K}[G]$ satisfies $yx=1_G$ for some $y\in\mathbb{K}[G]$, where $1_G$ denotes the identity of the group algebra $\mathbb{K}[G]$. We have by (\ref{eq:prj}) that
\begin{equation}
1_H=\pi_H(1_G)=\pi_H(yx)=\pi_H(y)x,
\end{equation}
so $x$ is left invertible in $\mathbb{K}[H]$ as well. The right invertible case can be treated analogously.
\end{proof}

For the proof of Theorem \ref{theorem:unit}, we need to use the known proof of Kaplansky's unit conjecture for groups with unique product property.

\begin{definition}
A group $G$ is said to be a unique product group if given two non-empty finite subsets $A$ and $B$ of $G$, there exist at least one element $g\in G$ which can be uniquely written as $g=ab$ for some $a\in A$ and $b\in B$.
\end{definition}

\begin{theorem}[$\cite{dsp}$, Theorem 26.2]\label{theorem:up}
If $G$ is a unique product group, then $\mathbb{K}[G]$ has only units of the form $\lambda\cdot g$ for some $\lambda\in\mathbb{K}^\times$ and $g\in G$.
\end{theorem}

\begin{remark}
The above theorem was originally stated in $\cite{dsp}$ for two unique products groups instead of unique product groups. A group $G$ is said to be a two unique products group if for any two non-empty finite subsets $A$ and $B$ of $G$ with $|A|+|B|>2$, there are at least two distinct elements $g,h\in G$ which can be uniquely represented as $g=ab$, $h=cd$, where $a,c\in A$ and $b,d\in B$. However, it was proved later by Strojnowski $\cite{as}$ that a group $G$ is a unique product group if and only if it is a two unique products group.
\end{remark}

It is known that every unique product group is torsion-free. Conversely, all torsion-free abelian groups are unique product groups.

\begin{proof}[Proof of Theorem \ref{theorem:unit}]
Let $x\in Z\left(\mathbb{K}[G]\right)$ be a central unit of $\mathbb{K}[G]$, and let $H_x$ be the subgroup of $G$ generated by $\mathrm{Supp}(x)$. It follows from Lemma \ref{lemma:inv} that $x$ is also a unit in $\mathbb{K}[H_x]$. By Proposition \ref{proposition:FC}, $H_x$ is a torsion-free FC-group, which is necessarily abelian by Lemma \ref{lemma:abel}. Since $H_x$ is a unique product group, we can apply Theorem \ref{theorem:up} to complete the proof.
\end{proof}


\begin{thebibliography}{99}
\fontsize{9pt}{9pt}\selectfont
\bibitem{ma1}M. Abouzaid, \textit{A cotangent fibre generates the Fukaya category}, Adv. Math. 228 (2011), 894-939.
\bibitem{ma2}M. Abouzaid, \textit{Symplectic cohomology and Viterbo's theorem}, in \textit{Free loop spaces in geometry and topology}, volume 24 of IRMA Lect. Math. Theor. Phys., pages 271-485. Eur. Math. Soc., Zurich, 2015.
\bibitem{aak}M. Abouzaid, D. Auroux, and L. Katzarkov, \textit{Lagrangian fibrations on blowups of toric varieties and mirror symmetry for hypersurfaces}, Publ. Math. IHES 123 (2016), 199-282.
\bibitem{as}M. Abouzaid and P. Seidel, \textit{An open string analogue of Viterbo functoriality}, Geom. Topol. 14 (2010), no. 2, 627-718.
\bibitem{agv}V.I. Arnold, S.M. Gusein-Zade, and A.N. Varchenko, \textit{Singularities of differentiable maps}, Vol. I, Monographs in Mathematics, vol. 82, Birkh\"{a}user Boston Inc., Boston, MA, 1985, The classification of critical points, caustics and wave fronts, Translated from the Russian by Ian Porteous and Mark Reynolds.
\bibitem{bfk}M. Ballard, D. Favero,and L. Katzarkov, \textit{A category of kernels for equivariant factorizations and its implications for Hodge theory}, Publ. Math. IHES. 120 (2014), 1-111.
\bibitem{abe}A. Beilinson, \textit{Coherent sheaves on $\mathbb{P}^n$ and problems in linear algebra}, Funktsional, Anal. i Prilozhen. 12 (1978), no. 3, 68-69.
\bibitem{bee}F. Bourgeois, T. Ekholm, Y. Eliashberg, \textit{Effect of Legendrian surgery}, With an appendix by S. Ganatra and M. Maydanskiy, Geom. Topol., 16, (2012), 301-389.
\bibitem{bo2}F. Bourgeois and A. Oancea, $S^1$-\textit{equivariant symplectic homology and linearized contact homology}, Int. Math. Res. Not., (13):3849-3937, 2017.
\bibitem{ct}A. C\u{a}ld\u{a}raru and J. Tu, \textit{Curved $A_\infty$ algebras and Landau-Ginzburg models}, New York J. Math. 19 (2013), 305-342.
\bibitem{cdgg}B. Chantraine, G. Dimitroglou-Rizell, P. Ghiggini, and R. Golovko, \textit{Geometric generation of the wrapped Fukaya category of Weinstein manifolds and sectors}, arXiv:1712.09126.
\bibitem{cg}R. Cohen and S. Ganatra, \textit{Calabi-Yau categories, the Floer theory of a cotangent bundle, and the string topology of the base}, preliminary version available at http://math.stanford.edu/~ralph/papers.html.
\bibitem{bd}B. Davison, \textit{Superpotential algebras and manifolds}, Adv. Math. 231(2), 879-912 (2012).
\bibitem{dl}L. Diogo and S. Lisi, \textit{Symplectic homology of complements of smooth divisors}, J. Topology 12 (2019) 967-1030.
\bibitem{td}T. Dyckerhoff, \textit{Compact generators in categories of matrix factorizations}, Duke Math. J. 159 (2011), no. 2, 223-274.
\bibitem{ekl}T. Ekholm and Y. Lekili, \textit{Duality between Lagrangian and Legendrian invariants}, arXiv:1701.01284.
\bibitem{egh}Y. Eliashberg, A. Givental and H. Hofer, \textit{Introduction to symplectic field theory}, Geom. Funct. Anal. 2000, Special Volume, Part II, 560-673.
\bibitem{etl1}T. Etg\"{u} and Y. Lekili, \textit{Koszul duality patterns in Floer theory}, Geometry and Topology 21 (2017) 3313-3389.
\bibitem{fss}K. Fukaya, P. Seidel and I. Smith, \textit{Exact Lagrangian submanifolds in simply-connected cotangent bundles}, Invent. Math. 172:1-27, 2008.
\bibitem{fu}M. Futaki and K.Ueda, \textit{Homological mirror symmetry for Brieskorn–Pham singularities}, Sel. Math. New Ser. 17, 435–452 (2011).
\bibitem{bg}B. Gammage, \textit{Mirror symmetry for Berglund-Hübsch Milnor fibers}, arXiv:2010.15570.
\bibitem{sg2}S. Ganatra, \textit{Symplectic cohomology and duality for the wrapped Fukaya category}, arXiv:1304.7312.
\bibitem{gps2}S. Ganatra, J. Pardon and V. Shende, \textit{Sectorial descent for wrapped Fukaya categories}, arXiv:1809.03427.
\bibitem{gp1}S. Ganatra and D. Pomerleano, \textit{A Log PSS morphism with applications to Lagrangian embeddings}, J. Topology (2021) 14(1), 291-268.
\bibitem{gp2}S. Ganatra and D. Pomerleano, \textit{Symplectic cohomology rings of affine varieties in the topological limit}, Geom. Funct. Anal. 30, 334-456 (2020).
\bibitem{gs}S. Ganatra and K. Siegel, \textit{On the embedding complexity of Liouville manifolds}, arXiv:2012.04627.
\bibitem{gg}G. Gardam, \textit{A counterexample to the unit conjecture for group rings}, Ann. of Math. 194 (2021) 967-979.
\bibitem{mg}M. Gromov, \textit{Pseudo holomorphic curves in symplectic manifolds}, Invent. Math. 82 (1985), 307-347.
\bibitem{hs}M. Habermann and J. Smith, \textit{Homological Berglund-Hübsch mirror symmetry for curve singularities}, J. Symp. Geom. (2020) 18 (6), 1515-1574.
\bibitem{hlw}Y. Han, X. Liu and K. Wang, \textit{Exact Hochschild extensions and deformed Calabi-Yau completions}, arXiv:1909.02200.
\bibitem{hwz}H. Hofer, K. Wysocki and E. Zehnder, \textit{Properties of pseudoholomorphic curves in symplectisations I: Asymptotics}, Ann. Inst. Henri Poincar\'{e}, 13 (1996), 337-379.
\bibitem{ak1}A. Keating, \textit{Lagrangian tori in four-dimensional Milnor fibers}, Geom. Funct. Anal. Vol. 25 (2015), 1822-1901.
\bibitem{bk}B. Keller, \textit{Deformed Calabi-Yau completions}, with an appendix by M. Van den Bergh, J. Reine Angew. Math. 654 (2011), 125-180.
\bibitem{bke}B. Keller, \textit{Erratum to ``Deformed Calabi-Yau completions"}, arXiv:1809.01126.
\bibitem{ol}O. Lazarev, \textit{Geometric and algebraic presentations of Weinstein domains}, arXiv:1910.01101.
\bibitem{lp1}Y. Lekili, A. Polishchuk, \textit{Arithmetic mirror symmetry for genus 1 curves with n marked points}, Selecta Math. (N.S.) 23 (2017), 1851-1907.
\bibitem{lu1}Y. Lekili and K. Ueda, \textit{Homological mirror symmetry for Milnor fibers via moduli of $A_\infty$ structures}, arXiv:1806.04345.
\bibitem{lu2}Y. Lekili and K. Ueda, \textit{Homological mirror symmetry for Milnor fibers of simple singularities}, arXiv:2004.07374, to appear in Algebraic Geometry.
\bibitem{yl1}Y. Li, \textit{Koszul duality via suspending Lefschetz fibrations}, J. Topology (2019) 12 (4), 1174-1245.
\bibitem{yl2}Y. Li, \textit{Exact Calabi-Yau categories and odd-dimensional Lagrangian spheres}, arXiv:1907.09257.
\bibitem{vl}V.A. Lunts, \textit{Categorical resolution of singularities}, J. Algebra 323 (2010), no. 10, 2977-3003.
\bibitem{bn}B. H. Neumann, \textit{Groups with finite classes of conjugate elements}, Proc. London Math. Soc. (3) 1 (1951), 178-187.
\bibitem{jo} J. \"{O}inert, \textit{Units, zero-divisors and idempotents in rings graded by torsion-free groups}, arXiv:1904.04847.
\bibitem{dsp}D. S. Passman, \textit{Infinite group rings}, Pure and Applied Mathematics, 6. Marcel Dekker, Inc., New York (1971).
\bibitem{pv}A. Polishchuk and A. Vaintrob, \textit{Matrix factorizations and singularity categories for stacks}, Ann. Inst. Fourier (Grenoble) 61:7 (2011), 2609-2642.
\bibitem{pva}A. Polishchuk and U. Varolgunes, \textit{On homological mirror symmetry for chain type polynomials}, arXiv:2105.03808.
\bibitem{ar}A.F. Ritter, \textit{Deformations of symplectic cohomology and exact Lagrangians in ALE spaces}, Geom. Funct. Anal. 20(3):779-816, 2010.
\bibitem{es}E. Segal, \textit{The closed state space of affine Landau-Ginzburg B-models}, J. Noncommut. Geom. 7 (2013), no. 3, 857-883.
\bibitem{ps1}P. Seidel, \textit{Fukaya categories and Picard-Lefschetz theory}, Z\"{u}rich Lect. in Adv. Math., European Math. Soc., Z\"{u}rich, 2008.
\bibitem{ps2}P. Seidel, \textit{A biased view of symplectic cohomology}, Current developments in mathematics, International Press, Somerville, MA, 2008, 211-253.
\bibitem{ps4}P. Seidel, \textit{Disjoinable Lagrangian spheres and dilations}, Invent math (2014) 197:299-359.
\bibitem{ps5}P. Seidel, \textit{Lectures on categorical dynamics and symplectic topology}, available at http://www-math.mit.edu/~seidel/.
\bibitem{ps6}P. Seidel, \textit{Suspending Lefschetz fibrations}, with an application to local mirror symmetry, Comm. Math. Phys. 297 (2010), 515-528.
\bibitem{ps9}P. Seidel, \textit{Graded Lagrangian submanifolds}, Bull. Soc. Math. France 128 (2000), no. 1, 103-149.
\bibitem{ps10}P. Seidel, \textit{A long exact sequence for symplectic Floer cohomology}, Topology 42 (2003), no. 5, 1003-1063.
\bibitem{ps11}P. Seidel, \textit{Fukaya $A_\infty$-structures associated to Lefschetz fibrations. I}, J. Symplectic Geom. 10(3): 325-388 (2012).
\bibitem{ss}P. Seidel and J. Solomon, \textit{Symplectic cohomology and q-intersection numbers}, Geom. Funct. Anal. 22, 443-477 (2012).
\bibitem{sp}\textit{Stacks Project}, available at: http://stacks. math. columbia. edu.
\bibitem{as}A. Strojnowski, \textit{A note on u.p. groups}, Comm. Algebra 8(3) (1980), 231-234.
\bibitem{zs}Z. Sylvan, \textit{On partially wrapped Fukaya categories}, J. Topology 12 (2019) 372-441.
\bibitem{dt}D. Tonkonog, \textit{From symplectic cohomology to Lagrangian enumerative geometry}, Adv. Math. 352, 717-776 (2019).
\bibitem{cv}C. Viterbo, \textit{Exact Lagrange submanifolds, periodic orbits and the cohomology of free loop spaces}, J. Differential Geom., 47(3):420-468, 1997.
\bibitem{ww}W. Wu, \textit{Exact Lagrangians in $A_n$-surface singularities}, Math. Ann. 359 (2014), no. 1-2, 153-168.
\bibitem{zz1}Z. Zhou, \textit{Symplectic fillings of asymptotically dynamically convex manifolds I}, J. Topology, 14(1):112-182, 2021.
\bibitem{zz2}Z. Zhou, \textit{Symplectic fillings of asymtotically dynamically convex manifolds II$-k$-dilations}, arXiv:1910.06132.
\end{thebibliography}
\end{document}